\documentclass[12pt]{article}
\usepackage{fancyhdr}
\usepackage{hyperref}
\usepackage{graphicx} % Required for inserting images
\usepackage{amsfonts}
\usepackage{amsthm,amssymb,amsmath}
\usepackage[T1]{fontenc}
\usepackage{mathabx}
\usepackage[utf8]{inputenc}
\usepackage{hyperref}
\usepackage{lipsum}
\usepackage{mathrsfs}
\usepackage{enumitem}
\usepackage{MnSymbol}
\usepackage{stmaryrd}
\usepackage{mathtools}
\usepackage{float}
\usepackage[all,cmtip]{xy}
\usepackage{tikz-cd}
\tikzcdset{row sep/normal=50pt, column sep/normal=50pt}
\newtheorem{lemma}{Lemma}[section]

\newtheorem{theorem}[lemma]{Theorem}
\newtheorem{example}[lemma]{Example}

\newtheorem{corollary}[lemma]{Corollary}

\usepackage{blindtext}
\usepackage[a4paper, total={6in, 8in}]{geometry}
\title{Isomorphism of Multiprojective Bundles and Projective Towers}
\author{Ashima Bansal, Supravat Sarkar, Shivam Vats}
\date{}
\begin{document}
\maketitle
\begin{abstract}
We study when two projective bundles over two arbitrary smooth projective varieties of different dimensions can be isomorphic. We show that two multi-projective bundles (fibre product of projective bundles) over different projective spaces cannot be isomorphic, except in the trivial case. We also give necessary and sufficient conditions for the top varieties of two height 3 towers of projective bundles being isomorphic, under certain assumptions.
\end{abstract}
\begin{center}
\end{center}
\begin{center}
\end{center}

\section{Introduction} 
 Projective spaces and multiprojective spaces (products of projective spaces) are the simplest examples of complete algebraic varieties. Relativising the constructions of projective and multi-projective spaces, one gets the notion of projective bundles and multi-projective bundles, the latter being the fibre product of projective bundles over the same base. With these constructions in mind, one is naturally led to study projective bundles and multi-projective bundles over projective spaces. A natural question in this area is when two multi-projective bundles over different projective spaces are isomorphic. In \ref{9}, this question is answered for projective bundles. One of our goals in this article is to completely answer this question for multiprojective bundles.

  In section 3, we study when two projective bundles over two arbitrary smooth projective varieties can be isomorphic. This question was studied in \ref{10} under some assumption on the lower bound of the dimension of the projective bundle in terms of the dimensions of the bases. We study this question without any assumption on the lower bound on the dimension of the projective bundle, but with the assumption that one of the vector bundles is trivial. Our result is:
\\\\
\textbf{Theorem A} (Theorem \ref{theorem:theorem 3.12}): Let $X $ and $ Y$ be smooth projective varieties of dimension $m$ and $n$ respectively. Let $E$ and $ F$ be vector bundles of rank $r+1$ and $s+1$ over $X$ and $Y$ respectively. Suppose   
$ \mathbb{P}_X(E)\cong \mathbb{P}_Y(F) $ and $ E$ is trivial. Then one of the following holds:
\begin{enumerate}
\item[(i)] $X\cong Y$ and F is trivial up to line bundle twists.
\item[(ii)] There exists a smooth projective variety $Z$ and two bundles $ E_{1}$ and $ F_{1} $ on $Z $ of ranks $s+1, r+1$ respectively, such that $F_{1}$ is trivial and 
$ X\cong \mathbb{P}_Z(E_{1})  $ and $Y = \mathbb{P}_Z(F_{1}) =Z\times \mathbb{P}^r  $.
\end{enumerate}

We also obtain several results describing when a multiprojective bundle has a special kind of product structure.

In section 4, we first describe the precise condition that is imposed on the vector bundles by the assumption that the Chow rings of the corresponding projectivisations are isomorphic. As a corollary, we give a partial topological analogue of the result in \ref{9}.  Then we answer the question we raised in the beginning about two multi-projective bundles over different projective spaces being isomorphic. Our result is:
\\\\
\textbf{Theorem B} (Theorem \ref{theorem:theorem 4.4})
\textit{Let $m\neq n$ be positive integers, $E_{1},...,E_{r}$ be vector bundles on $\mathbb{P}^m$ and $F_{1},...,F_{s} $  vector bundles on $ \mathbb{P}^n$ of rank at least two. Let $p_{i}, q_{j}$ be the positive integers such that $\mathrm{rk}(E_{i})= p_{i}+ 1 $ and $\mathrm{rk}(F_{j})= q_{j}+ 1 $ for all $i, j$. Then the following are equivalent:
\begin{enumerate}
\item[(i)] $\mathbb{P}_{\mathbb{P}^m}(E_{1},...,E_{r}) \cong 
\mathbb{P}_{\mathbb{P}^n}(F_{1},..., F_{s})$;
\item[(ii)] $r=s,  \{ m, p_{1},...,p_{r} \}= \{ n, q_{1},...,q_{r} \} $ as multisets and all $ E_{i}$ and $F_{j} $ are trivial bundles up to line bundle twists.
\end{enumerate}}

A projective tower of height $m$ is a sequence
\begin{equation*}
C_{m}\xlongrightarrow{\pi_{m}}C_{m-1}\xlongrightarrow{\pi_{m-1}}...\xlongrightarrow{\pi_{2}}C_{1}\xlongrightarrow{\pi_{1}}C_{0} = \{point\}
\end{equation*}
where $C_{i} = \mathbb{P}(\xi_{i-1})$ is the projectivization of a vector bundle $\xi_{i-1}$ over $C_{i-1}$. Complex projective towers were studied from a topological point of view in \ref{5}, where, in low dimensions, authors gave necessary and sufficient conditions for the top manifolds of two complex projective towers to be homeomorphic. To the best of our knowledge, the only analogous work in algebraic geometry done so far is in \ref{9}, for height 2 projective towers. At the end of section 4, we extend this work to the next stage, for height 3 projective towers, under certain assumptions. Our result is the following:
\\\\
\textbf{Theorem C} (Theorem \ref{theorem:theorem 4.5})
\textit{Let $m,n,r$ be positive integers with $m<n$. Let $ E_{1} $ and $F_{1}$  be vector bundles of rank $r+1$ over $ \mathbb{P}^m,\mathbb{P}^n$ respectively and   $ E_{2},F_{2}$ be vector bundles of rank $ n+1,m+1$ over $ \mathbb{P}_{\mathbb{P}^{m}}(E_{1}), \mathbb{P}_{\mathbb{P}^{n}}(F_{1})$ respectively. Suppose $ \mathbb{P}_{\mathbb{P}(E_{1})}(E_{2}) \cong \mathbb{P}_{\mathbb{P}(F_{1})}(F_{2})$. Then:  
\begin{enumerate}
\item[(i)] If $ r \neq m,n $ then all $ E_{i}$'s, $F_{j}$'s are trivial up to  line bundle twists. 
\item[(ii)] If $ r=m $, then $ E_{2}, F_{1}$ are trivial upto line bundle twists. There is a trivialization of $F_1$ such that if $ \mathbb{P}(F_{1})= \mathbb{P}^n \times \mathbb{P}^m \xlongrightarrow{ \pi} \mathbb{P}^m$  denotes the second  projection, then $ F_{2} = \pi^{\ast} E_{1} \bigotimes L$ for some line bundle $L$ on $ \mathbb{P}^n \times  \mathbb{P}^m .$
\item[(iii)] If $ r=n $ then $(ii)$ holds with $m$ replaced by n and $E_{i}$ replaced by $F_{i}$.\end{enumerate}}
\noindent
\section{Notations and Conventions}
\begin{enumerate}
\item Let $k$ be an algebraically closed field of any characteristic. By a variety we mean an integral separated $k$-scheme of finite type over $k$. By a point on a variety, we always mean a closed point. Any isomorphism of varieties is an isomorphism over $k$. The variety Spec $k$ is denoted by pt (abbreviation of "point").
\item Let $A^{\ast}(X)$ denote the Chow ring of a smooth projective variety $X$. We identify Pic$(X)=A^1(X)$, where by Pic(X) we mean the Picard group of $X$. For a vector bundle $E$ over a smooth projective variety $X$, $c(E)\in A^{\ast}(X)$ denotes the total Chern class of $E$. For definitions and results about the Chow ring and Chern classes see \ref{1}, \ref{2}.
\item We follow the convention of the projective bundle as in  [\ref{8}, chapter 2, section 7]. We shall use the following results frequently:
\begin{enumerate}
\item[(a)] If $p:\mathbb{P}(E)\to X $ is a projective bundle, then $p_*\mathcal{O}_{\mathbb{P}(E)}=\mathcal{O}_X.$
\item[(b)] If $E, F$ are vector bundles of rank $n+1$ over a smooth projective variety $X$, and $\phi:\mathbb{P}(E)\to\mathbb{P}(F)$ an isomorphism over $X,$ then $\phi$ is an isomorphism of the corresponding PGL$_{n+1} -$bundles, and $E\cong F\otimes L$ for some line bundle $L$ on $X.$
\item[(c)](See \ref{11}, Appendix A). If $E$ is a vector bundle of rank $r+1$ on a smooth projective variety $X$, then under the natural pullback we have,$\vspace{1mm}$ $A^{\ast}(\mathbb{P}(E))\cong {A^{\ast}(X)[u]}/(\sum_{i = 0}^{r+1}(-1)^{i}{c_{i}(E)u^{r+1-i}})$ as $A^{\ast}(X)$-algebra. The isomorphism is given $\vspace{1mm}$ by $[\mathcal{O}_{\mathbb{P}(E)}(1)]\longleftarrow u$.
\end{enumerate}
\item The following fact follows from the formula for the Chern class of tensor product of a vector bundle and a line bundle (for example, see \ref{1}, example 3.2.2 or \ref{12}, chapter 1).

\textbf{Fact:} Let $E$ be a rank $r+1$ vector bundle on $\mathbb{P}^n$, and $c(E)=(1+ac_1(\mathcal{O}_{\mathbb{P}^{n}}(1)))^{r+1}$ for some integer $a$, then $c(E\otimes \mathcal{O}_{\mathbb{P}^{n}}(-a))=1.$

\item Let $ E_{1},..., E_{r}$ be vector bundles over a variety $X$. We denote the fibre product $ \mathbb{P}(E_{1}) \times_{X} \mathbb{P}(E_{2}) \times_{X}...\times_{X }\mathbb{P}(E_{r})$ by $ \mathbb{P}_{X}(E_{1},...,E_{r})$, and we call this a multiprojective bundle over $X$. When there is no ambiguity we shall not write the subscript $X$.
\item For a free abelian group $A$ of finite rank we denote the rank of $A$ by rk$\medspace A$. If $E$ is a vector bundle over a variety $X$ then rk$\medspace E$ denotes the rank of $E$.
\item For a smooth projective variety $X$, we denote the Brauer group of $X$ by Br($X$).
\item We say that a vector bundle $E$ over a variety $X$ is trivial up to (or with) a line bundle twist if there is a line bundle $L$ on $X$ with $E\otimes L$ trivial vector bundle.
\item If $P(X_1,..., X_n)$ is a polynomial with complex coefficients, we denote by deg$_{X_i}P$ the $X_i$-degree of $P$. The degree of the zero polynomial is $-\infty$ by convention.
\end{enumerate}
\begin{section}{Projective bundles over smooth projective varieties}
We shall use the following lemmas.
\begin{lemma}
Let $f:X\longrightarrow Y$ be a proper morphism of varieties such that $f_{\ast}\mathcal{O}_{X} = \mathcal{O}_{Y}$. Let $E$ be a vector bundle of rank $n$ over $Y$. If $f^{\ast}E$ is trivial up to line bundle twist, so is $E$.
\end{lemma}
\begin{proof}
Let $f^{\ast}E = \mathcal{O}_{X}^{\bigoplus n}\otimes L = L^{\bigoplus{n}}$, where $L$ is a line bundle over $X$. Applying $f_{\ast}$ and using projection formula we get $E\otimes f_{\ast}\mathcal{O}_{X} = (f_{\ast}L)^{\bigoplus{n}}$. Since $f_{\ast}\mathcal{O}_{X} = \mathcal{O}_{Y}$, we get $(f_{\ast}L)^{\bigoplus{n}} = E$. Now $f_{\ast}L$ is a direct summand of vector bundle $E$, so $f_{\ast}L$ must itself be a vector bundle and on comparing the ranks we see that $f_{\ast}L$ must be a line bundle. So, $E = f_{\ast}L\otimes\mathcal{O}_{Y}^{\bigoplus n}$ is trivial up to line bundle twist.    
\end{proof}
\begin{lemma}
Let $X$, $Y$ be varieties and $f: X\longrightarrow Y$ a proper morphism. Suppose that the scheme-theoretic fibre $f^{-1}\{y\}$ is isomorphic to  $\mathrm{Spec}\medspace k$, for all $\medspace y$ in $Y$. Then $f$ is an isomorphism.    
\end{lemma}
\begin{proof}
As $f$ is proper and quasi-finite, so $f$ is finite by Zariski's Main Theorem. Let $U =$ Spec$\medspace A$ be an affine open subset in $Y$ and $f^{-1}U =$ Spec$\medspace B.$  We have a ring inclusion $A\subset B$ induced by $f$ and $B$ is a finitely generated $A$-module.
We need to show $A = B$, it suffices to show $A_{\mathfrak{m}} = B_{\mathfrak{m}},$ for all maximal ideals $\mathfrak{m}$ of $A$. By our assumption, $\mathfrak{m}B$ is a maximal ideal of $B$ hence $\mathfrak{\mathfrak{m}}B + A = B$. 
 After localizing we get $\mathfrak{m} A_{\mathfrak{m}}.B_{\mathfrak{m}}+A_{\mathfrak{m}} = B_{\mathfrak{m}}$ so $A_{\mathfrak{m}} = B_{\mathfrak{m}}$ by Nakayama's Lemma as $B_{\mathfrak{m}}$ is a finitely generated $A_{\mathfrak{m}}$-module.

\end{proof}

\begin{lemma}
Let $X$ be a projective variety, and let $\mathbb{P}^{n}\xlongrightarrow{f} X$ be a surjective map. Then either $X$ is a point or the dimension of $X$ is n.
\end{lemma}
\begin{proof}
Let $H$ be a very ample line bundle on $X$. Then $f^\ast H$ is a globally generated line bundle. So either $f^\ast H = \mathcal{O}_{\mathbb{P}^{n}}$, in this case $f$ is constant so $X =$ pt, or $f^\ast H = \mathcal{O}_{\mathbb{P}^{n}}(k)$ for some $k>0$, hence both $H$ and $f^\ast H$ are ample, so $f$ is a finite map and dim$\medspace X$ is $n$.
\end{proof}

\begin{lemma}
If $X$ is a smooth projective variety and there is a nonconstant surjective map 
$\mathbb{P}^{n}\xlongrightarrow{f} X$, then $X$ is isomorphic to $\mathbb{P}^n$.
\end{lemma}
\begin{proof}
See \ref{3}, where Lazarsfeld proved this using Mori theory.
\end{proof}

\begin{lemma}
Let $X$ and $Y$ be smooth projective varieties and $X\times\mathbb{P}^{n}\xlongrightarrow{f}Y$ a morphism such that $f_{\ast}\mathcal{O}_{X\times\mathbb{P}^{n}} = \mathcal{O}_{Y}$. Then there exist varieties $Z$ and $T$ such that $Y = Z\times T$ and  morphisms $f_{1}: X\longrightarrow Z$ and $f_{2}: \mathbb{P}^{n}\longrightarrow T$ such that $f = f_{1}\times f_{2}$. 
\end{lemma}
\begin{proof}
See \ref{4}, Proposition 5.
\end{proof}
\begin{lemma}\label{lemma:lemma 3.6}
Let $X$ and $Y$ be smooth projective varieties such that $\mathbb{P}^{n}\times X$  
is isomorphic to  $\mathbb{P}^{n}\times Y$. Then $X$ is isomorphic to $Y$.
\end{lemma}
\begin{proof}
See \ref{4}, Theorem 6.
\end{proof}
\begin{lemma}\label{lemma:lemma 3.7} Let $m,n,r$ and $s$ be positive integers. Then:
\begin{enumerate}
\item[(i)] Suppose $m\neq n$. Let $E$ and $F$ be vector bundles of rank $r+1$ and $s+1$ over $\mathbb{P}^{m}$ and $\mathbb{P}^{n}$ respectively. If \medspace $\mathbb{P}(E)\cong\mathbb{P}(F)$ then $r=n$, $s=m$, and $E$ and $F$ are trivial upto line bundle twists. Also, if $p:\mathbb{P}(E)\to \mathbb{P}^m $,  $p:\mathbb{P}(F)\to \mathbb{P}^n $ are the projections and $\phi:\mathbb{P}(E)\to\mathbb{P}(F)$ is an isomorphism, then the map $\Phi=(p, q\circ\phi): \mathbb{P}(E)\to \mathbb{P}^{m}\times\mathbb{P}^{n}$ is an isomorphism.

\item[(ii)] Let $ E, F $ be rank $ r+1$ vector bundles over $ \mathbb{P}^n $. If $\mathbb{P}(E) \cong \mathbb{P}(F)$ then there is an automorphism  $ \psi: \mathbb{P}^n \longrightarrow \mathbb{P}^n $ and $ L\in \mathrm{Pic}(\mathbb{P}^n)$ such that $ E\cong \psi^{\ast}F\bigotimes L$ as vector bundles over $\mathbb{P}^n $.
\end{enumerate}
\end{lemma}
\begin{proof}
The last statement of $(i)$ follows from Theorem A of \ref{9}. One can see that $\mathbb{P}_{\mathbb{P}^m}(E)\xlongrightarrow{\Phi}\mathbb{P}_{\mathbb{P}^m}(\mathcal{O}_{\mathbb{P}^n}^{n+1})$ is an isomorphism of PGL$_{n+1} -$ bundles, we get $E$ is trivial with line bundle twist. Similarly, $F$ is trivial with line bundle twist.

For $(ii)$, note that given an isomorphism of $\mathbb{P}(E)$ and $\mathbb{P}(F)$, $\mathbb{P}(E)$ gets two projective bundle structures: $p$ and $q\circ\phi$ in the notation of $(i)$. If they are not different projective bundle structures in the notion (P) of \ref{9}, (1) implies (3) part in lemma 1.5 of \ref{9} (we do not need any assumption on the characteristic of $k$ as our projective bundles are Zarisky locally trivial) gives the conclusion in $(ii)$. If the projective bundle structures are different in notion (P) of \ref{9}, then Theorem A of \ref{9} and a similar argument as in $(i)$ shows either both $E$ and $F$ are trivial up to line bundle twists, or both $E$ and $F$ are tangent bundles of $\mathbb{P}^n $ upto line bundle twists. In both cases, the conclusion of $(ii)$ clearly holds.
\end{proof}
We shall also need a relative version of Lemma $3.7.(i)$.
\begin{lemma}\label{lemma:lemma 3.8}
Let $X$ be a variety, $E_{1}$ and $F_{1}$ be vector bundles over $X$ of rank $m+1$ and $n+1$ such that $m\neq n$, $\mathbb{P}(E_{1})\xlongrightarrow{p}X$ and  $\mathbb{P}(F_{1})\xlongrightarrow{q}X$ be the projectivation of $E_{1}$ and $F_{1}$ respectively. Let $E_{2}$ and $F_{2}$ be vector bundles of rank  $r+1$ and $s+1$ on $\mathbb{P}(E_{1})$ and $\mathbb{P}(F_{1})$ repectively. If $\mathbb{P}_{\mathbb{P}(E_{1})}(E_{2})\cong\mathbb{P}_{\mathbb{P}(F_{1})}(F_{2})$ over $X$, then $r=n$, $s=m$ and there exist $L\in\mathrm{Pic}(\mathbb{P}(E_{1}))$ and $M\in\mathrm{Pic}(\mathbb{P}(F_{1}))$ such that $E_{2} = p^{\ast}F_{1}\otimes L$ and $F_{2} = q^{\ast}E_{1}\otimes M$.
\end{lemma}

\begin{proof}

\noindent
($\Rightarrow$) Consider the following diagram,\vspace{10mm}

\hspace{35mm}
\xymatrix{\mathbb{P}_{\mathbb{P}(E_{1})}(E_{2}) \ar[d]  & \cong & {\mathbb{P}_{\mathbb{P}(F_{1})}(F_{2})} \ar[d] \\  \mathbb{P}(E_{1})  \ar[rd]_p   &  &  \mathbb{P}(F_{1})\ar[ld]^q \\ & X}
\vspace{10mm}
\\
It is clear that the situation over each point $x\in X$ is exactly the situation in Lemma 3.7(i). So, $ r = n $ and $s = m$. 
Let $Z =  \mathbb{P}_{{\mathbb{P}(E_{1})}}(E_{2})$. The above diagram gives the map $Z\xlongrightarrow{\Phi}\mathbb{P}(E_{1})\times_{X}\mathbb{P}{(F_{1})}$ and for each point $x\in X$, by Lemma 3.7(i) the map $\Phi_{x}: Z_{x}\longrightarrow(\mathbb{P}(E_{1})\times_{X}\mathbb{P}(F_{1}))_{x}$ is an isomorphism. So, all scheme-theoretic fibres of $\Phi$ are Spec $k$. By Lemma 3.2, $\Phi$ is an isomorphism. We have the following commutative diagram,

\begin{center}
\begin{tikzcd}
\mathbb{P}_{\mathbb{P}(E_{1})}(E_{2})\cong{Z} \arrow[r, "\Phi" ]  \arrow[d]
& \mathbb{P}(E_{1})\times_{X}\mathbb{P}(F_{1})\cong\mathbb{P}_{\mathbb{P}(E_{1})}(p^{\ast}F_{1}) \arrow[ld] \\
\mathbb{P}(E_{1})
\end{tikzcd}
\end{center}
\vspace{2mm}
So, $\Phi$ is an isomorphism of PGL$_{n+1} -$bundles over $\mathbb{P}(E_{1})$.
This shows $E_{2} = p^{\ast}F_{1}\otimes L$ for some $L\in$ Pic$(\mathbb{P}(E_{1}))$. Similarly, $F_{2} = q^{\ast}E_{1}\otimes M$ for some $M\in$ Pic$(\mathbb{P}(F_{1}))$. 
\end{proof}
Our first result of this section is the following,
\begin{theorem}\label{theorem:theorem 3.9}
Let $X $ and $ Y$ be smooth projective varieties of dimension $m$ and $n$ respectively. Let $E$ and $ F$ be vector bundles of rank $r+1$ and $s+1$ over $X$ and $Y$  respectively with $ r+m = s+n \geq m+n $. Suppose $ \mathbb{P}(E)\cong \mathbb{P}(F)$. Then one of the following holds:
\begin{enumerate}
\item[(i)] $m = n$ and there is an isomorphism $X\xlongrightarrow{\psi} Y $ such that $E\cong\psi^{\ast}F\otimes L$ for some line bundle L on $X$.
\item[(ii)] $X\cong \mathbb{P}^m, Y \cong \mathbb{P}^n$, $r=n, s=m $ and $E, F$ are trivial bundles with line bundle twists.
    
\end{enumerate}
\end{theorem}
\begin{proof}
As $ r+m = s+n \geq m+n $ we have $r\geq n$  and $s\geq m$. Let $ \mathbb{P}(E) \xlongrightarrow{p} X ,$   $ \mathbb{P}(F) \xlongrightarrow{q} Y $ be the  projectivization of $E,F$  and  $ \phi: \mathbb{P}(E) \longrightarrow \mathbb{P}(F)$ be an isomorphism. Suppose $(i)$ does not hold. We shall show that (ii) holds. Let $D_{z}= p^{-1}(z)$ be the fibre over $z\in X$.

Suppose $(q\circ\phi)|_{D_{z}}$ is constant, say $\psi(z)$, for all $z\in X$. Using local sections of the morphism $p$, one easily sees that $\psi$ is a morphism of varieties. For $y\in Y$, we have $p^{-1}\psi^{-1} (y)\cong q^{-1}(y)\cong \mathbb{P}^s$ via $\phi$. So $\psi^{-1} (y)$ is a smooth projective variety. Since $p^{-1}\psi^{-1} (y)$ is a $\mathbb{P}^r$-bundle over $\psi^{-1} (y)$, a comparison of Picard groups shows $\psi^{-1} (y)\cong Spec(k)$, for all $y\in Y$. So by Lemma 3.2, $X\xlongrightarrow {\psi}Y$ is an isomorphism and the following diagram commutes,
\begin{center}
\begin{tikzcd}
\mathbb{P}(E) \arrow[r, "\phi"   ] \arrow[d, "p"]
&  \mathbb{P}(F)\arrow[d, "q", swap ] \\
X\arrow[r, "\psi"]
& |[, rotate=0]|  Y
\end{tikzcd}
\end{center}
Hence, $\mathbb{P}(E)\cong\mathbb{P}(F)\times_YX = \mathbb{P}(\psi^{\ast}F)$ over $X$, so $E\cong\psi^{\ast}F\otimes L$ for some line bundle $L$ on $X$. So, $(i)$ holds a contradiction.

So, $ q \circ  \phi \mid_{D_{z}}: D_{z} \longrightarrow Y  $ is  nonconstant for some $z\in X$.  Using Lemma 3.3 and the fact that $ D_{z} \cong \mathbb{P}^r $  we get dim($ q \circ \phi (D_{z})$) $=r$ hence, $n \geq r$. This, together with our assumption $r\geq n $ gives $ r=n$ and $s=m $. Also, $ q \circ \phi : D_{z} \longrightarrow Y $  is surjective, so by Lemma 3.4 we get $ Y \cong \mathbb{P}^{n}$ and similarly, $X\cong \mathbb{P}^m$. Then using Lemma 3.7, we see that $(i)$ or $(ii)$ holds. Since $(i)$ does not hold so $(ii)$ holds. 
\\ 
\end{proof}

The above Theorem is not true for $ r+m = s+n \less m+n $, here is a class of counterexamples.
\\
\begin{example}
 Let $Z $ be any smooth projective variety that is not a point and let $ E_{1}$ and $F_{1}$ be vector bundles of different ranks ($\geq2$) on $Z$. Let $ X=\mathbb{P}(E_{1}),$ $ Y= \mathbb{P}(F_{1})$. Let $\pi_{1}:X\to  Z $ and $\pi_2:Y\to Z $ be the natural projections.  Let $E= \pi_{1}^{\ast} F_{1}$ and $ F = \pi_{2}^{\ast} E_{1} $, then  we have $ \mathbb{P}_X(E) \cong \mathbb{P}_Y(F) \cong  \mathbb{P}_Z(E_{1}) \times_{Z} \mathbb{P}_Z(F_{1}) $. But clearly (i) in Theorem 3.9 cannot hold as dim $X=$ dim $Z$ $+$ rk $E_1-1\neq$ dim $Z $ $+$ rk $F_1-1=$ dim $Y$, and (ii) in Theorem 3.9 also cannot hold as $Z$ is not a point.
\end{example}

In the above class counterexamples, we see that $X$ and $Y$ are projective bundles over a common base. This suggests the question: If $m\neq n$, $r+n = s+n < m+n,$ and $\mathbb{P}_{X}(E)\cong\mathbb{P}_{Y}(F)$ (with the same notation as in Theorem 3.9), are $X$ and $Y$ projective bundles over a common base?

 The answer is no in general as the following example shows. 
\begin{example}[\ref{6}, Chapter 1, Section 3.1]
Let $X = \mathrm{Gr(2, n+1)}$, the grassmannian of lines in $\mathbb{P}^{n}$ and $Y = \mathbb{P}^{n}$. Let $\mathbb{F}_{n} = \{(x, \ell)\hspace{2mm}| \hspace {2mm} \ell\hspace{2mm} is\hspace{2mm} a\hspace{2mm} line\hspace{2mm} in \hspace{2mm}\mathbb{P}^{n},\hspace{1mm} x\in\ell \}$ be a flag variety. Let $V$ be the tautological 2-bundle over $X$, i.e.,
$V = \{(\ell, v)\in X\times k^{n+1}\hspace{2mm}|\hspace{2mm}v\in E_{l}\}$, where $E_{l}$ is the 2-dimensional subspace of $k^{n+1}$ corresponding to the line $l$ in $\mathbb{P}^n$. The natural projection $\mathbb{F}_{n}\xlongrightarrow{p}X$ and $\mathbb{F}_{n}\xlongrightarrow{q}Y$ identifies $\mathbb{F}_{n} \cong \mathbb{P}_{X}(V^{\ast})$ and $\mathbb{F}_{n}\cong\mathbb{P}_{\mathbb{P}^{n}}(\Omega_{\mathbb{P}^{n}}) = \mathbb{P}_{\mathbb{P}^{n}}(T_{\mathbb{P}^{n}}^{\ast})$.  

So, $\mathbb{P}_{Y}(\Omega_{\mathbb{P}^{n}})\cong\mathbb{P}_{X}(V^\ast)$. But $X$ and $Y$ are not projective bundles over a common base, this can be seen for example by using $\mathrm{Pic(Gr(2, n+1))} \medspace=\medspace \mathbb{Z}$.
\end{example}\vspace{2mm}

[\ref{10}, Theorem 2] gives an affirmative answer to our question, when $r+m=s+n=m+n-1$.
Now we show that even after relaxing the conditions on $r,m,s$ and $n$ one still has an affirmative answer if either $E$ or $F$ is trivial. The following is the main result of this section.
\begin{theorem}\label{theorem:theorem 3.12}
Let $X $ and $ Y$ be smooth projective varieties of dimension $m$ and $n$ respectively. Let $E$ and $ F$ be vector bundles of rank $r+1$ and $s+1$ over $X$ and $Y$ respectively. Suppose   
$ \mathbb{P}_X(E)\cong \mathbb{P}_Y(F) $ and $ E$ is trivial. Then one of the following holds:
\begin{enumerate}
\item[(i)] $X\cong Y$ and F is trivial up to line bundle twist.
\item[(ii)] There exists a smooth projective variety $Z$ and two bundles $ E_{1}$ and $ F_{1} $ on $Z $ of ranks $s+1, r+1$ respectively, such that $F_{1}$ is trivial and 
$ X\cong \mathbb{P}_Z(E_{1})  $ and $Y = \mathbb{P}_Z(F_{1}) =Z\times \mathbb{P}^r  $.
\end{enumerate}
\end{theorem}
\begin{proof}Let $\phi: \mathbb{P}_{X}(E) \longrightarrow \mathbb{P}_{Y}(F)$ be an isomorphism, and on comparing dimensions of $\mathbb{P}_X(E)$ and $\mathbb{P}_Y(F)$ we get $r+m=s+n$.
\noindent
 We know that $E$ is trivial, so $\mathbb{P}(E)= X \times \mathbb{P}^r$. Applying Lemma 3.5 to $ q \circ \phi: X \times \mathbb{P}^r \longrightarrow Y $ we get varieties $ Z,T$  such that $ Y= Z\times T $ and morphisms $ X \xlongrightarrow {\pi_{1}} Z $, $ \mathbb{P}^r \xlongrightarrow { \pi} T$ which satisfies $ q \circ \phi = \pi_{1} \times \pi $. Since $ q \circ \phi $ is onto, so $\pi_{1}$, $ \pi $ are onto. As $Y$ is smooth projective, so $Z$ and $T$ are smooth projective. By Lemma 3.4, either $ T =$ pt or $ T= \mathbb{P}^r $.   
\vspace{2pt}
\\ \\
{\textbf{Case 1}}: $ T= pt$ so $ Y= Z $. Hence we have $\psi=\pi_1: X\longrightarrow Y$ such that the following diagram commutes,
\begin{center}
\begin{tikzcd}
\mathbb{P}(E) \arrow[r, "\phi"   ] \arrow[d,"p"]
&  \mathbb{P}(F)\arrow[d , swap,"q"] \\
X\arrow[r, "\psi"]
& |[, rotate=0]|  Y
\end{tikzcd}
\end{center}
By the same argument as in Theorem 3.9, $\psi$ is an isomorphism. So we have $\mathbb{P}(F)\cong \mathbb{P}((\psi^{-1})^*E)$ over $Y$. So, $F\cong (\psi^{-1})^*E\otimes L$ for some line bundle $L$. As $E$ is trivial we get $F$ is trivial with line bundle twist.\\\\
{\textbf{Case 2}}: If $T= \mathbb{P}^r$ then $ Y= Z\times \mathbb{P}^r = \mathbb{P}_{Z}(F_{1}) $ where $F_{1}=\mathcal{O}_{Z}^{\oplus r+1} $. The situation is summarized in the following commutative diagram:
\begin{center}
\begin{tikzcd}
X \times \mathbb{P}^{r} = \mathbb{P}(E) \arrow[r,swap]{ur}{\sim}[swap]{\phi} \arrow[d, "p"] \arrow[rd, "\pi_1 \times \pi"]
& \mathbb{P}(F) \arrow[d, "q"]   \\
X  & Y = Z \times \mathbb{P}^{r} .
\end{tikzcd}
\end{center}
For $(z,a)\in Z\times \mathbb{P}^r,$ we have $\mathbb{P}^s\cong q^{-1}(z,a)\cong (q\circ\phi)^{-1}(z,a)=\pi_1^{-1}(z)\times \pi^{-1}(a)$. As dim $Z=$ dim $X-s$ so, dim $\pi_1^{-1}(z)\medspace\geq s$ for all $z$ in $Z$. Hence, for all $a\in \mathbb{P}^r $, $\pi^{-1}(a)$ must be isomorphic to Spec $k$. So $\pi$ is an isomorphism, and we can assume without loss of generality that $\pi$ is the identity map on $\mathbb{P}^r$. Let $a\in \mathbb{P}^r $ be a point and
$$ i: Z \longrightarrow Z \times \mathbb{P}^r,\,\
  z \mapsto (z,a) $$
  $$j: X \longrightarrow X \times \mathbb{P}^r,  \,\ 
x \mapsto (x,a) $$\\be the inclusions. We have the fibre square \\ 
\begin{center}
\begin{tikzcd}
X \arrow[r, hook, "j"  ] \arrow[d,]
& X\times  \mathbb{P}^r \arrow[d,  ] \\
Z \arrow[r, hook, "i"]
& |[, rotate=0]| Z \times  \mathbb{P}^r
\end{tikzcd}.
\end{center}
Since $ X \times \mathbb{P}^r  \cong \mathbb{P}_{Y}(F) $ and $ Z \times \mathbb{P}^r = Y $ we have following fibre diagram \\ 
\begin{center}
\begin{tikzcd}
X \arrow[r, hook,   ] \arrow[d,]
& \mathbb{P}_{Y}(F) \arrow[d,  ] \\
Z \arrow[r, hook, "i"]
& |[, rotate=0]| Y
\end{tikzcd}.
\end{center}

It follows that $X \cong \mathbb{P}_Z(E_{1})$ where $ E_{1}= i^{\ast} F $.

\end{proof}

Now we show that if $X=Y$, (i) in the above theorem always holds.
\begin{corollary} \label{lemma: lemma 3.13}Let $X$ be a smooth projective variety and $F$ be a vector bundle of rank $s+1$, where $s\geq 1$. If $\mathbb{P}_{X}(F)\cong X\times\mathbb{P}^{s}$ as varieties, then $F$ is trivial up to line bundle twist.    
\end{corollary}
\begin{proof}
The proof is by induction on dim $X$. For dim $X=0$, there is nothing to show. Now assume  dim$\medspace X \geq 1$. We follow the proof of Theorem 3.12 and the notations introduced there, with $Y=X$ and $r=s$. If we are in case 1, we are done. Now suppose we are in case 2, which will give us a smooth projective variety $Z$ with dim$\medspace Z \less$ dim$\medspace X$, and $ Z\times \mathbb{P}^s= X \cong \mathbb{P}_{Z}(E_{1})$ for some vector bundle $E_1$ of rank $s+1$ on $Z.$ We have $\pi_1: X \cong \mathbb{P}_{Z}(E_{1})\to Z$ be the projection, and we again assume without loss of generality that $\pi:\mathbb{P}^s\to \mathbb{P}^s$ is the identity map. Since $ Z\times \mathbb{P}^s \cong \mathbb{P}_{Z}(E_{1})$,  $E_{1}$ is trivial up to line bundle twists by induction. So we have a commutative diagram,
\begin{center}
\begin{tikzcd}
 X   \arrow[rr, "\phi_1"] \arrow[swap]{dr}{\pi_{1}} &  _{\sim}  & Z\times\mathbb{P}^{s}\arrow{dl}{\pi_2} \\[5pt]
 & Z
\end{tikzcd}
\end{center}
where $\pi_2$  is the projection onto the first factor. This gives a commutative diagram,

\begin{center}    
\begin{tikzcd}
\mathbb{P}_{X}(F)\arrow[d, "q"]\arrow[r, "\phi^{-1}"]\arrow[dr, phantom ]&  X\times\mathbb{P}^{s} \arrow[r, "\phi_1\times\medspace \text{id}"] \arrow[d, "\pi_1\times\medspace \text{id}"] & Z\times \mathbb{P}^s\times\mathbb{P}^s \arrow[d, "\pi_2 \times \text{id}"]\\
X \arrow[r, "\text{id}"] & Z\times \mathbb{P}^s \arrow[r, "\text{id}"] & Z\times \mathbb{P}^s
\end{tikzcd}
\end{center}
Note that $ (\pi_2 \times \text{id})( z,\alpha , \beta) = (z,\beta)$ for all $ z\in Z, \alpha ,\beta \in \mathbb{P}^s $. This shows that $ \mathbb{P}_{X}(F)\cong X\times \mathbb{P}^s$ over $X$, where these are regarded as projective bundles over $X$ via the canonical projections to $X$. Hence $F$ is trivial up to line bundle twist. The proof is complete by induction.  
\end{proof}

\begin{corollary}\label{lemma:lemma 3.14}
Let $m, p_{1},...,p_{r}$ be natural numbers. Let $ E_{1},E_{2},...,E_{r}$ be vector bundles over $ \mathbb{P}^m $ with  $\mathrm{rk} (E_{i})= p_{i}+1.$ Suppose there is a variety $Z$ such that $ \mathbb{P}_{\mathbb{P}^m}(E_{1},E_{2},..,E_{r})$ isomorphic to $\mathbb{P}^m \times Z $. Then one of the following holds:

\begin{enumerate}
\item[(i)] There exists an integer $i$ in $\{1,...,r\}$ such that $E_{i}\cong \mathcal{O}_{\mathbb{P}^m}^{m+1}\otimes L$ 
for some line bundle $L$ over $\mathbb{P}^{m}$, and
$Z\cong\mathbb{P}_{\mathbb{P}^{m}}(E_{1},...,E_{i-1},E_{i+1},...,E_{r})$.
\item[(ii)] all $ E_{i} $'s are trivial up to line bundle twists and $ Z \cong 
\mathbb{P}^{p_{1}}\times ...\times  \mathbb{P}^{p_{r}}$.
\end{enumerate}   
\end{corollary}
\begin{proof}
We will prove this by induction on $ r$. For $ r=1$, we have $$ \mathbb{P}_{\mathbb{P}^m}(E_{1}) \cong \mathbb{P}^m\times Z = \mathbb{P}_{Z}(\mathcal{O}_Z^{m+1}).$$ By Theorem 3.9, either $Z = \mathbb{P}^{p_{1}}$ (and hence $E_{1}$ is trivial up to line bundle twists by Corollary 3.13) or there is an isomorphism $\mathbb{P}^{m}\xlongrightarrow{\psi} Z$, such that $E_{1} = \psi^{\ast}\mathcal{O}_Z^{m+1}\otimes L$ for some line bundle $L$. In both cases, we are done. 

Now let $r\geq 2$. If all the $E_{i}$'s are trivial up to line bundle twists, then $ Z\times \mathbb{P}^m\cong \mathbb{P}^{p_{1}}\times ...\times \mathbb{P}^{p_{r}}\times \mathbb{P}^m$, and then using Lemma 3.6, we get $ Z \cong \mathbb{P}^{p_{1}}\times ...\times \mathbb{P}^{p_{r}}.$ So suppose some $E_{i}$, say $E_{r}$, is not trivial up to a line bundle twist and also assume none of $E_{1},..., E_{r}$ is of the form $\mathcal{O}_{\mathbb{P}^m}^{m+1}\otimes L$ because otherwise we are done again by Lemma 3.6. We have $$\mathbb{P}_{Z}( \mathcal{O}_Z^{m+1})\cong \mathbb{P}_{\mathbb{P}^m}(E_{1},E_{2},...,E_{r})\cong \mathbb{P}_{\mathbb{P}_{\mathbb{P}^m}(E_{1},...,E_{r-1})}(\pi^{\ast} E_{r}) $$ where 
 $ \pi : \mathbb{P}_{\mathbb{P}^m} ( E_{1},...,E_{r-1})\longrightarrow \mathbb{P}^m $ is the projection. We have either $(i)$ or $(ii)$ of Theorem 3.12 holds, with $X=Z,$ $Y=\mathbb{P}_{\mathbb{P}^m} ( E_{1},..., E_{r-1})$. If $(i)$ holds $ \pi^{\ast}E_{r}$ is trivial up to a line bundle twist. So  $E_{r}$ is trivial up to line bundle twist by Lemma 3.1, which is a contradiction. Hence, $(ii)$ of Theorem 3.12 holds. So there is a smooth projective variety $ Z_{1}$   such that $ \mathbb{P}_{\mathbb{P}^m}(E_{1},..., E_{r-1}) \cong \mathbb{P}^m \times Z_{1}$, now as none of the $ E_{i} $'s is $ \mathcal{O}_{\mathbb{P}^m}^{m+1}\bigotimes L$, we have $ E_{1},..., E_{r-1}$ are trivial up to line bundle twists, by induction. Since we assumed that none of $E_{1},..., E_{r}$ is of the form $\mathcal{O}^{m+1}\otimes L$, we see that $ p_{i} \neq m$  for all $ i  \less r $. 
 
 We have, $ \mathbb{P}^m \times Z \cong \mathbb{P}_{\mathbb{P}^m}(E_{1},...,E_{r}) \cong \mathbb{P}_{\mathbb{P}^m}(E_{r}) \times \mathbb{P}^{p_{1}} \times ...\times \mathbb{P}^{p_{r-1}}$. Since $ m\neq p_{i}\hspace{2mm}$  for all $i\less r $, applying Theorem 3.12 repeatedly we get $ \mathbb{P}_{\mathbb{P}^m}(E_{r}) \times \mathbb{P}^{p_{1}} \times ...\times \mathbb{P}^{p_{i}}\cong \mathbb{P}^m \times W_i  $ for smooth projective varieties $W_i $, for each $0\leq i<r$ (use reverse induction on $i$). So, $ \mathbb{P}_{\mathbb{P}^m}(E_{r})\cong \mathbb{P}^m \times W_0  $. By the base case of Corollary 3.14 (which we already proved),  we get $E_{r} $  is trivial up to line bundle twists, which contradicts our assumption. Now by induction, we complete the proof.
\end{proof}
 \begin{corollary}
Let $ p_{1},...,p_{r}$ be distinct integers and $X, Z_{1},..., Z_{r}$ be smooth projective varieties with $ X\cong \medspace\mathbb{P}^{p_{i}}\times Z_{i}$ for all  $i$. Then there exists a smooth projective variety $W$ such that $ X\cong \medspace\mathbb{P}^{p_{1}}\times ...\times \mathbb{P}^{p_{r}} \times W .$
 \end{corollary}

\begin{proof}
 We will prove this by induction on $r$. For $r=1 $, there is nothing to show. For $ r\geq 2$, we have $$ \mathbb{P}_{Z_{r}}( \mathcal{O}_{Z_r}^{p_{r}+1})= \mathbb{P}^{p_{r}}\times Z_{r} \cong \mathbb{P}^{p_{i}}\times Z_{i} \cong \mathbb{P}_{Z_{i}}( \mathcal{O}_{Z_i}^{p_{i}+1})$$ for all  $i \less r$ . Since $p_{i}\neq p_{r} $, Theorem 3.12 shows there exist smooth projective varieties $W_{i}$ such that $ Z_{r}\cong \mathbb{P}^{p_{i}}\times W_{i}$ for all $i\less r$.   By induction, there exists a smooth projective variety $W$  such that $ Z_{r}\cong \mathbb{P}^{p_{1}}\times ...\times \mathbb{P}^{p_{r-1}} \times W$. So finally we have, $$ X\cong \mathbb{P}^{p_{r}}\times Z_{r}\cong \mathbb{P}^{p_{1}}\times ...\times \mathbb{P}^{p_{r-1}} \times \mathbb{P}^{p_{r}} \times W . $$
\end{proof}
\begin{corollary}
    Let $ r\geq 1$, $ 0\leq \ell \leq r $ and $ p_{1},p_{2},...,p_{r},m\geq1 $ be integers such that $ p_{1}=p_{2}=...=p_{\ell}=m$ and $ p_{i}\neq m \medspace$ for all  $i\gtr \ell .$ Let $Z$ be smooth projective variety and $F $ be a vector bundle of rank $m+1$ over $Z$. Suppose that $\mathbb{P}^{p_{1}}\times ...\times \mathbb{P}^{p_{r}} \cong \mathbb{P}_{Z}(F).$ Then $\ell \geq 1$, $ Z = \mathbb{P}^{p_{2}} \times...\times 
    \mathbb{P}^{p_{r}} $, $F $ is trivial with line bundle twist.
\end{corollary}

\begin{proof}
We will prove by induction on $r$. For $r=1$ it is clear. Now assume $r\geq 2 $. We have $ \mathbb{P}_{\mathbb{P}^{p_{1}}\times... \times \mathbb{P}^{p_{r-1}}}(\mathcal{O}_{Z}^{p_{r}+1})\cong \mathbb{P}_{Z}(F)$ so, by Theorem 3.12 one of the following cases occurs:\\\\
${\textbf{Case 1:}}$  $Z \cong \mathbb{P}^{p_{1}}\times ...\times \mathbb{P}^{p_{r-1}}$, $ F \cong \mathcal{O}_Z^{p_{r}+1} \bigotimes L$ for some line bundle $L$. In this case $p_r=m$, so we must have $\ell=r $ and $p_{1}=p_{2}=...=p_{r}=m $, hence we are done.\\\\
${\textbf{Case 2:}}$ $Z = \mathbb{P}^{p_{r}}\times Z_{1} $, $\mathbb{P}^{p_{1}}\times ...\times \mathbb{P}^{p_{r-1}} $ is a $ \mathbb{P}^m$- bundle over $Z_{1} $ for some smooth projective variety $Z_{1}$. By induction, we have $ \ell \geq 1$, $ Z_{1} = \mathbb{P}^{p_{2}} \times ...\times \mathbb{P}^{p_{r-1}}.$ So, $ Z= \mathbb{P}^{p_{2}} \times ...\times \mathbb{P}^{p_{r}} $ and $ Z  \times \mathbb{P}^m \cong \mathbb{P}_{Z}(F).$ Now Corollary 3.13 shows $ F$ is trivial up to line bundle twist. Now induction completes the proof.
\end{proof}

This is the case $s=1$ of the following more general theorem. 
\begin{theorem}\label{theorem:theorem 3.17}
Let $ r\geq 1$ and $ p_{1},p_{2},...,p_{r}\geq 1 $ be integers such that $\mathbb{P}^{p_{1}}\times ...\times \mathbb{P}^{p_{r}} \cong \mathbb{P}_{Z}(F_{1},F_{2},...,F_{s}),$ where $Z$ is a smooth projective variety and $F_{i}$'s are vector bundles on $Z $ of rank $q_{i}+1 \medspace$ for all $i$. Then $ s\leq r$, and up to a permutation of the $p_i$'s we have $ p_{i}=q_{i}$ for $1 \leq i\leq s$, $ Z \cong \mathbb{P}^{p_{s+1}}\times ...\times \mathbb{P}^{p_{r}} $ and $ F_{i}$'s are trivial up to line bundle twists.      
\end{theorem}
\begin{proof}
We prove by induction on $s$. For $s=1$, the statement is true by Corollary 3.16. Now assume $ s\geq 2  $. Write $ \mathbb{P}_{Z}(F_{1},F_{2},...F_{s})= \mathbb{P}_{{\mathbb{P}_{Z}}(F_{2},...,F_{s})}(\pi^{\ast}F_{1})$ where $ \pi :\mathbb{P}_{Z}(F_{2},...,F_{s})\longrightarrow Z $  is the projection. By Corollary 3.16, $ \pi^{\ast} F_{1} $  is trivial up to line bundle twist (so $F_{1}$ is trivial up to line bundle twist by lemma 3.1), and up to permutation of $ p_{i}$'s we have $p_{1}= q_{1}$, $\mathbb{P}_{Z}(F_{2},..., F_{s}) = \mathbb{P}^{p_{2}}\times ...\times \mathbb{P}^{p_{r}} $. Now by induction, we are done. 
\end{proof}

\textbf{Remark:} Theorem 3.17 follows easily from looking at the elementary contractions of the Fano variety $\mathbb{P}^{p_{1}}\times ...\times \mathbb{P}^{p_{r}}$. We give our proof mainly to illustrate how Theorem 3.12 can be used to give an alternate proof of this result.
\noindent
\end{section}

\begin{section}{Chow rings, multiprojective bundles, towers of projective bundles} 
Corollary 4.3 is an interesting partial topological analogue of the main result of \ref{9}.
Theorem 4.4 is a substantial generalization of the main result of \ref{9} in the case of distinct dimensions. Theorem 4.5 tells us when the top varieties in two height 3 projective bundle towers are isomorphic, under certain assumptions.  
\\\\
We begin with a preliminary lemma.
\begin{lemma} \label{lemma:lemma 4.1}
For positive integers $n_i$'s and $m_j$'s, we have  $A^{\ast}(\prod_{i=1}^{r} \mathbb{P}^{n_{i}})$ is isomorphic $A^{\ast}(\prod_{i=1}^{s} \mathbb{P}^{m_{i}}) $ as graded groups if and only if $ r=s$ and $\{ m_{1},...,m_{s}\}= \{ n_{1},...,n_{s}\}$ as multisets.
\end{lemma}

\begin{proof}
    $ (\Leftarrow)  $ Obvious.\\ 
\\
$( \Rightarrow)$ Look at the Poincare series,
$$ \sum_{\ell}\textrm {rk}\medspace A^{\ell}( \prod_{i=1}^{r}\mathbb{P}^{n_{i}}) t^{\ell}=  \prod_{i=1}^{r} (\sum_{\ell} \textrm{rk}\medspace A^{\ell}(\mathbb{P}^{n_{i}}) t^{\ell})  = \prod_{i=1}^ { r}( 1+t+t^{2}+...+t^{n_{i}}).$$
So we get, 
\begin{equation}
    \prod_{i=1}^ { r}( 1+t+t^{2}+...+t^{n_{i}})= \prod_{j=1}^ { s}( 1+t+t^{2}+...+t^{m_{j}}).
\end{equation}
Looking at the coefficient of $t $ we get $ r=s $. Multiplying (1) by $ (1-t)^r $ we get, 
      \begin{equation}
          \prod_{i=1}^{r}(1-t^{n_{i}+1}) = 
          \prod_{i=1}^{r}(1-t^{m_{i}+1}). 
      \end{equation}
We claim that (2) $ \Rightarrow \{n_{1},...,n_{r} \} = \{ m_{1},...,m_{r}\}$ as multisets. By induction it suffices to show that min $\{ n_{1},...,n_{r}\} = $ min $\{ m_{1},...,m_{r}\} $. This follows by looking at the smallest integer $ d \gtr 0$ such that the coefficient of $ t^{d}$ is nonzero.

\end{proof}

We now describe a necessary and sufficient condition on vector bundles such that the Chow rings of the corresponding projective bundles are isomorphic.
\begin{theorem}\label{theorem:theorem 4.2}
Let $m\neq n$ be natural numbers  and $E$ and $F$ be vector bundles of ranks $r+1$ and $s+1$ on $\mathbb{P}^{m}$ and $\mathbb{P}^{n}$ respectively. The rings $A^{\ast}(\mathbb{P}(E))$  and  $A^{\ast}(\mathbb{P}(F))$  are isomorphic as graded rings if and only if $r=n$, $s=m$ and there exist L $\in \mathrm{Pic}(\mathbb{P}^m) $ and M $\in \mathrm{Pic}(\mathbb{P}^n)$ such that $ c(E\otimes L)=1 $ and $ c(F\otimes M)=1 $.
\end{theorem} 
    
\begin{proof}
($\Leftarrow$) We know, $$A^{\ast}(\mathbb{P}(E))\cong A^{\ast}(\mathbb{P}(E\otimes L)) \cong \frac{\mathbb{Z}[X,Y]}{(X^{m+1},Y^{n+1})} \cong A^{\ast}(\mathbb{P}(F\otimes M))\cong A^{\ast}(\mathbb{P}(F)).$$
\vspace{3mm}
\\
\vspace{2pt}($\Rightarrow$) Let $A^{\ast}(\mathbb{P}(E)) \cong A^{\ast}(\mathbb{P}(F))$ as graded rings. As $A^{\ast}(\mathbb{P}(E)) \cong A^{\ast}(\mathbb{P}^m\times \mathbb{P}^r)$ and $A^{\ast}(\mathbb{P}(F)) \cong A^{\ast}(\mathbb{P}^n\times \mathbb{P}^s)$ as graded groups so,$\medspace$ $ A^{\ast}(\mathbb{P}^m\times \mathbb{P}^r)\cong A^{\ast}(\mathbb{P}^n\times \mathbb{P}^s)$ as graded groups. By Lemma (4.1), we get $ \{ m,r\}= \{ n,s \}$ as multisets and as $m\neq n$, we get  $r=n$  and $s=m $.

Now we will show the second part. Without loss of generality let us assume $m<n$. Let  $ \mathbb{P}(E) \xrightarrow{p} \mathbb{P}^m $ and $ \mathbb{P}(F) \xrightarrow{q} \mathbb{P}^n $  be the projectivization of $E$  and $F$ and $A^{\ast}(\mathbb{P}(E))\xrightarrow{\psi} A^{\ast}(\mathbb{P}(F))$ be an isomorphism of graded rings. We have $p^{\ast}(\mathcal{O}_{\mathbb{P}^m}(1))$ and $ \mathcal{O}_{\mathbb{P}(E)}(1)\in A^{1}(\mathbb{P}(E))$ and similarly $q^{\ast}(\mathcal{O}_{\mathbb{P}^n}(1))$ and $ \mathcal{O}_{\mathbb{P}(F)}(1)\in A^{1}(\mathbb{P}(F))$.
Let us denote $t=q^{\ast}(\mathcal{O}_{\mathbb{P}^n}(1))$, $ u= \mathcal{O}_{\mathbb{P}(F)}(1)$, $x= \psi(p^{\ast}(\mathcal{O}_{\mathbb{P}^m}(1))$, $y= \psi(\mathcal{O}_{\mathbb{P}(E)}(1))$, so $t,u,x,y\in A^{1}(\mathbb{P}(F))$.

Let $a_i,$ $b_i$ be integers such that $c_i(F)=b_i c_1(\mathcal{O}_{\mathbb{P}^n}(1))^i$, $c_i(E)=a_i c_1(\mathcal{O}_{\mathbb{P}^m}(1))^i$ for each $i$. Since $c_i(E)=0$ for $i>m$, we can choose $a_i=0$ for $i>m$. Let $f(Y,X)=\sum_{i=0}^{n+1}a_iY^{n+1-i} (-X)^i$ and $g(U,T)=\sum_{i=0}^{m+1}b_iU^{m+1-i} (-T)^i$. So, $f(Y,X)\in \mathbb{Z}[Y,X]$, $g(U,T)\in \mathbb{Z}[U,T]$ are homogeneous polynomials of degree $n+1$, $m+1$ respectively, such that $Y^{n+1}$ and $U^{m+1}$ have coefficients $1$ in $f,g$ respectively, and $q^*c(F)=g(1,-t),$ $\psi (p^*c(E))=f(1,-x).$  Also, by the formula of Chow ring of projective bundle as in point 3. (c)
 of notations and conventions section, $$ A^{\ast}(\mathbb{P}(F))\cong \frac{\mathbb{Z}[T,U]}{(T^{n+1},g(U,T))} $$ via $$ t \longmapsfrom T $$ 
$$ u\longmapsfrom U$$ and $$ A^{\ast}(\mathbb{P}(F))\cong \frac{\mathbb{Z}[X,Y]}{(X^{m+1},f(Y,X))} $$ via $$ x \longmapsfrom X $$ 
$$ y\longmapsfrom Y.$$ Here the last isomorphism is obtained by first applying the formula of Chow ring of projective bundle to get $A^\ast(\mathbb{P}(E))$ and then using the isomorphism $\psi$. 

It is clear that  $\{t,u\}$ and $\{x,y\}$ are both basis of  $A^{1}(\mathbb{P}(F))$. So there is a matrix $A= \begin{pmatrix}
    a & b \\
    c & d
\end{pmatrix}
\in\mathrm{Gl_{2}}(\mathbb{Z})$ such that
\begin{gather}
\begin{pmatrix}
    x \\
    y
\end{pmatrix}
=
A
\begin{pmatrix}
    t \\
    u
\end{pmatrix}.
\end{gather}
 If $b =0$ then $a=\pm 1$ so $x=\pm t$, hence $ t^{m+1}=0$, but this is false as $t^{j+1}\neq 0$ for all $j \less n$. So $b$ must be nonzero.
 
Since coefficient of $U^{m+1}$ in $g$ is $1$, we can write, 
\begin{equation}
 (b U+ a T)^{m+1}= b^{m+1} g(U,T) + R(U,T)
\end{equation}
where $R(U,T)$ is a linear combination of $T U^m$, $T^2 U^{m-1}$ ,..., $T^{m+1}$. We have $x^{m+1}=0$, so $(b u +a t)^{m+1}=0$. Also $g(u,t)=0$, so we get $R(u,t)=0$. But the set \{$t u^m$, $t^2 u^{m-1}$,...,$t^{m+1}$\} is linearly independent in $A^{\ast}(\mathbb{P}(F))$ as  $m \less n$. This shows that $R(U, T)=0$. From (4) we have $ g(U,T)= (U + \frac{a}{b}T)^{m+1}$ so $\frac{a}{b}\in \mathbb{Z}$. (This forces $b=\pm 1$, since gcd$(a,b)=1$ for $A\in GL_2(\mathbb{Z})$). We also get $c(F \otimes \mathcal{O}_{\mathbb{P}^n}(\frac{a}{b}))= 1$, by the fact in point 4 of notations and conventions section. So, letting $M= \mathcal{O}_{\mathbb{P}^n}(\frac{a}{b}) $ we have $c(F \otimes M)=1 $.

We have, \begin{gather}
\begin{pmatrix}
    x \\
    y
\end{pmatrix}
=
A
\begin{pmatrix}
    t \\
    u
\end{pmatrix}
\end{gather} so

\begin{gather}
\begin{pmatrix}
    t \\
    u
\end{pmatrix}
=
A^{-1}
\begin{pmatrix}
    x \\
    y
\end{pmatrix}.
\end{gather}
Clearly,

\begin{gather*}
A^{-1}
=
\pm {\begin{pmatrix}
    d & -b \\
    -c & a
\end{pmatrix}}
\end{gather*} 
We can write, $$ (Y-\frac{d}{b}X )^{n+1} = f(Y,X)+ R_{1}(Y,X)  $$ where $R_{1}(Y,X) $ a linear combination of $Y^{n}X$, $Y^{n-1} X^{2}$,..., $ X^{n+1}$. Since we have $ (y-\frac{d}{b}x )^{n+1}=\pm t^{n+1}=0 $, $ f(y,x)=0 $ then we must have $R_{1}(y,x)=0$. So $X^{m+1} \mid R_{1}(Y,X) $ as \{$y^{n}x$, $y^{n-1} x^{2}$,..., $ y^{n-m+1}x^{m}$\} is linearly independent in $A^{\ast}(\mathbb{P}(F))$. So $ f(Y,X)\equiv(Y-\frac{d}{b}X )^{n+1}$ (mod $X^{m+1}$). Since we have chosen $a_i=0$ for $i>m$ we obtain $ f(Y,X)=(Y-\frac{d}{b}X )^{n+1}$. It follows that $c(E \otimes \mathcal{O}_{\mathbb{P}^m}(\frac{-d}{b}))= 1$. If we let $L= \mathcal{O}_{\mathbb{P}^m}(\frac{-d}{b}) $ then $c(E \otimes L)=1 $.
\end{proof}
 \noindent
$\textbf{Remark:}$ The observation that $b=\pm 1$, which we made in the proof above, gives a proof of Lemma 3.7.$(i)$. This is essentially the idea of the proof in \ref{9}.
We next state a topological application of the Theorem. This is the (partial) topological analogue of Lemma 3.7.$(i)$. 
\begin{corollary}
Let $E,F $ be topological vector bundles of rank $ r+1, s+1$ over $ \mathbb{P}^m, \mathbb{P}^n $ respectively, with $ m \less n$ . If $ \mathbb{P}(E), \mathbb{P}(F)$ are homotopy equivalent, then:
\begin{enumerate}
\item[(i)]{ $ r=n , s=m $.}
\item[(ii)] {$ E $ is trivial up to line bundle twist.}
\item[(iii)]{ $ F \bigoplus \mathcal{O}_{\mathbb{P}^{n}}^ {n-m-1}$ is trivial up to line bundle twist, so $F$ is stably trivial. }
\end{enumerate} 
Moreover if $ n\leq 3 $, then F is trivial up to line bundle twist.
\end{corollary}
 \begin{proof}
Since   $ \mathbb{P}(E)$ and $\mathbb{P}(F)$ are homotopy equivalent so we have cohomology rings $ H^{\ast}(\mathbb{P}(E))\cong H^{\ast}(\mathbb{P}(F)) $. The cohomology rings have an analogous description to the description of Chow rings. The notations and proof of Theorem 4.2 carry over to the topological setting with the Chow ring replaced by the cohomology ring. So we get $ r=n , m=s $ and we can twist $ E, F$ by line bundle to assume that $ c(E)= 1, c(F)=1 $. Since $ r+1= n+1 \gtr m$ so by [\ref{6}, Chapter 1, Section 6]  $ E, F \bigoplus \mathcal{O}_{\mathbb{P}^{n}}^{n-m-1}$ are trivial bundles. 

 If $ m=1, n=3 $ then again by [\ref{6}, Chapter 1, Section 6] either $F$ is trivial or $F$ is the unique nontrivial rank 2 vector bundle over $ \mathbb{P}^3 $ with $ c(F)=1 $. Now by [\ref{5}, Proposition 5.4] we have  $$ \pi_{6} ( \mathbb{P}(F)) \ncong \pi_{6}( \mathbb{P}^3 \times \mathbb{P}^1) = \pi_{6}(\mathbb{P}(E)) .$$ So $ \mathbb{P}(E), \mathbb{P}(F) $  are not  homotopy equivalent, a contradiction, so $F$ must be trivial.

 \end{proof}

The authors do not know whether one can show that $F$ is also trivial in the above corollary, for any $m<n$.

Now our goal is to generalize the main result of \ref{9} in the case of distinct dimensions (that is, Lemma 3.7(i)), by replacing a single projective bundle with any finite fibre product of projective bundles. \\

The result is the following.

\begin{theorem} \label{theorem:theorem 4.4} 
Let $m\neq n$ be positive integers, $E_{1},...,E_{r}$ be vector bundles on $\mathbb{P}^m$ and $F_{1},...,F_{s} $  vector bundles on $ \mathbb{P}^n$ of rank at least two. Let $p_{i}, q_{j}$ be the positive integers such that $\mathrm{rk}(E_{i})= p_{i}+ 1 $,  $\mathrm{rk}(F_{j})= q_{j}+ 1 $ for all $i, j$. Then the following are equivalent:
\begin{enumerate}
\item[(i)] $\mathbb{P}_{\mathbb{P}^m}(E_{1},...,E_{r}) \cong 
\mathbb{P}_{\mathbb{P}^n}(F_{1},..., F_{s})$;
\item[(ii)] $r=s,  \{ m, p_{1},...,p_{r} \}= \{ n, q_{1},...,q_{r} \} $ as multisets and all $ E_{i}$ and $F_{j} $ are trivial bundles up to line bundle twists. 
\end{enumerate}  
\end{theorem}
\begin{proof}
$(ii)\Rightarrow (i)$ It is clear that $$\mathbb{P}_{\mathbb{P}^m}(E_{1},...,E_{r}) \cong \medspace\mathbb{P}^m \times \mathbb{P}^{p_{1}}\times ....\times \mathbb{P}^{p_{r}} \cong \mathbb{P}_{\mathbb{P}^n}(F_{1},...,F_{r})  $$ 
$(i)\Rightarrow (ii) $ We have $$ A^{\ast}(\mathbb{P}_{\mathbb{P}^{m}}(E_{1},...,E_{r}))\cong A^{\ast}( \mathbb{P}^{m}\times\prod_{i=1}^{r}\mathbb{P}^{p_{i}})$$ and $$A^{\ast}(\mathbb{P}_{\mathbb{P}^{n}}(F_{1},...,F_{s}))\cong A^{\ast}(\mathbb{P}^{n}\times\prod_{i=1}^{s}\mathbb{P}^{q_{i}} ) $$
as graded groups. By Lemma \ref{lemma:lemma 4.1}  we get $r=s$ and $ \{ m, p_{1},...,p_{r} \}= \{ n, q_{1},...,q_{r} \}$ as multisets.
\noindent
We now prove the remaining part of the theorem by induction on $r$. For $r=1 $ it is true by Lemma 3.7(i). Now suppose $r\gtr 1 $. Without loss of generality assume that $m<n,$ $ p_r=n,$ $p_{1}=p_{2}=...=p_{\ell-1}=m$, $ p_{i}\neq m $ for all $i\geq\ell  $, $q_{r-\ell+1}=...=q_{r}=m$, $ q_{i}\neq m$ for all $i\leq r-\ell$. Here $ 1\leq \ell \leq r$ is an integer.

Let $ X= \mathbb{P}(F_{1},...,F_{r})$ and $ \phi: \mathbb{P}(E_{1},..,E_{r}) \longrightarrow \mathbb{P}(F_{1},...,F_{r}) $ be an isomorphism. Let $ x= \mathcal{O}_{\mathbb{P}^m}(1),$ $ t =  \mathcal{O}_{\mathbb{P}^n}(1),$ $ u_{i}=\mathcal{O}_{\mathbb{P}(F_{i})}(1) $ all considered as elements of  Pic$(X) = A^{1}(X) $ via pullback by natural projections and $ \phi $.

As in the proof of Theorem \ref{theorem:theorem 4.2}, there are $ g_{i}(U_i,T)\in \mathbb{Z}[U_{i},T] $ homogenous polynomials of degree $ q_{i}+1$ with $  U^{q_{i}+1}$ having coefficient 1, such that $g_i(1,-t)=c(F_i)$ (considered as an element of $A^{\ast}(X)$ via pullback by natural projection). Also, considering $X$ as a projective bundle tower and applying the formula of Chow ring of projective bundle repeatedly, we get $$ A^{\ast}(X) \cong \frac{\mathbb{Z}[T, U_{1},...,U_{r} ]}{( T^{n+1}, g_{1}(U_{1},T),...,g_{r}(U_{r},T))}$$ via $$ t \longmapsfrom T $$ 
$$ u_{i}\longmapsfrom U_{i}$$ 

Since $\{t,u_1,u_2,...,u_r\}$ is a basis of the free abelian group $A^1(X)$, there are unique integers $a,b_{1},b_{2},...,b_{r}  $ such that $$ x= at+ \sum_{i=1}^{r} b_{i}u_{i}.$$\\ Let $X=\mathbb{P}(F_{1},...,F_{r})$ and $ \pi: X \xlongrightarrow{\phi^{-1}}\mathbb{P}(E_{1},...,E_{r}) \longrightarrow \mathbb{P}^m$ be the composition. Let $ X\xlongrightarrow{\pi_{1}}\mathbb{P}(F_{1},...,F_{r-1}) $ be the projection. For $z\in \mathbb{P}(F_{1},...,F_{r-1})$, under the map $\pi_{1}^{-1}(z) (\cong \mathbb{P}^m)\xlongrightarrow{\pi} \mathbb{P}^{m}$ the line bundle $\mathcal{O}_{\mathbb{P}^{m}}(1)$ pulls back to $\mathcal{O}_{\pi_1^{-1}(z)}(b_{r})$. So, $\mathcal{O}_{\pi_1^{-1}(z)}(b_{r})$ is globally generated. This implies $b_{r}\geq 0$. If $b_{r}= 0$, then $\pi$ factors as $ X\longrightarrow \mathbb{P}(F_{1},...,F_{r-1}) \xlongrightarrow{\pi'} \mathbb{P}^m $, where the first map is the natural projection. \\\\
\noindent
{{\textbf{Claim 1}}}: There exists $i$ with $ r-\ell+1 \leq i\leq r $ such that $b_{i}\neq 0$.
\begin{proof}
Suppose $b_{i}=0 $ for all $r-\ell+1 \leq i\leq r $. We have just shown $\pi$ factors as $ X\longrightarrow \mathbb{P}(F_{1},...,F_{r-1}) \xlongrightarrow{\pi'} \mathbb{P}^m $, where the first map is the natural projection. Using the same argument repeatedly (using induction to be more precise), we get:
 there is $\tau :\mathbb{P}(F_{1},...,F_{r-\ell})\longrightarrow \mathbb{P}^m $, such that $\pi$ factors as $ X\longrightarrow \mathbb{P}(F_{1},...,F_{r-\ell}) \xlongrightarrow{\tau} \mathbb{P}^m $, where the first map is the natural projection. For $w\in \mathbb{P}^m$, we have  $ \pi^{-1}(w)\cong \prod_{i=1}^{r} \mathbb{P}^{p_{i}} $  and  $ \pi^{-1}(w)\cong \mathbb{P}_{\tau^{-1}(w)}( q_1^{\ast}F_{r-\ell+1}\mid_{\tau^{-1}(w)},...,q_1^{\ast}F_{r}\mid_{\tau^{-1}(w)})$, where $q_1:\mathbb{P}(F_{1},...,F_{r-\ell})\to \mathbb{P}^n$ is the projection. By Theorem \ref{theorem:theorem 3.17} we get $ p_{i}=m $  for at least $\ell$ many $i$'s,  which is a contradiction.   
\end{proof}

Without loss of generality assume that $b_{r}\neq 0.$ Since $b_r\geq0$, we get $b_r>0.$\\ \\
{\textbf{Claim 2:} $b_{i} = 0$  for all $i<r $ and $b_{r} = 1$. 
\begin{proof}
Let $\underline{U}$ denotes the vector $ (U_1,...,U_{r})$. Write $$(aT+\sum_{i = 1}^{r}b_{i}U_{i})^{m+1} = \sum_{i = 1}^{r} Q_{i}(\underline{U}, T)g_{i}(U_{i},T)+ R(\underline{U},T) $$ where $Q_{i}g_i, R\in \mathbb{Z}[\underline{U},T]$ are homogeneous polynomials of degree $m+1$ and deg$_{U_{i}}R\leq q_{i}$ for all $i$. Since $R$ is homogeneous of degree $m+1\leq n$ we have deg$_{T}R\leq n$.

We have $(at+\sum_{i = 1}^{r} b_{i}u_{i})^{m+1} = x^{m+1} = 0$ and $g_{i}(u_{i}, t) = 0$ for all $i$. So we get $R(\underline{u},t) = 0$. 
 Since $\{t^{\alpha }{u_{1}^{\alpha_{1}}...}u_{r}^{\alpha_{r}}| \medspace 0\leq \alpha\leq n \,\ {and} \,\ \medspace 0\leq\alpha_{i}\leq q_{i}\}$ is linearly independent, we get $R = 0$.
 
So, 
\begin{equation*}
(aT + \sum_{i = 1}^{r}b_{i}U_{i})^{m+1} = \sum_{i =1}^{r}Q_{i}(\underline{U} , T )g_i(\underline{U}, T). 
\end{equation*}
Putting $T = 0$ in the above equation we get,
\begin{equation*}
(\sum_{i = 1}^{r}b_{i}U_{i})^{m+1} = \sum_{i = 1}^{r}Q_{i}(\underline{U}, 0)U_{i}^{q_{i}+1}    
\end{equation*}

For $i<r$, if $b_{i} \neq 0$ then the coefficient of $U_{i}U_{r}^{m}$ in the left-hand side of the above equation is non-zero (as $b_{r}\neq 0$),  but the coefficient of $U_{i}U_{r}^{m}$ in the right-hand side of above equation is 0 as $q_{i}\geq 1$ for all $i$ and $q_{r} = m$. So $b_{i} = 0$ for all $i<r $.

Since $x$ is part of a basis of $A^{1}(X)$, we get gcd$(a , b_{r}) = 1$.
Since $q_r=m$ we can write
\begin{equation*}
(b_{r}U_{r} + aT)^{m+1} = b_{r}^{m+1}g_{r}(U_{r},T) + R_{1}(U_{r}, T)
\end{equation*}
\noindent
where deg$_{U_{r}}R_{1}\leq m$. Clearly deg$_{T}R_{1} \leq n$ and $R_{1}(u_{r}, t) = 0$, this forces $R_{1} = 0$.

So, 
\begin{equation*}
(b_{r}U_{r}+aT)^{m+1} = b_{r}^{m+1}g_{r}(U_{r},T)   
\end{equation*}
$\Rightarrow$
\begin{equation*}
(U_{r} + \frac{a}{b_{r}}T)^{m+1} = g_{r}(U_{r},T)   
\end{equation*}
Since the above polynomial is in $\mathbb{Z}[U_{r},T]$, so $(\frac{a}{b_{r}})^{m+1}$, which is the coefficient of $T^{m+1}$ in $g_{r}(U_{r}, T)$, belongs to $\mathbb{Z}$.
So, $b_{r}\mid a$ and since gcd$(a, b_{r}) = 1$ we get $b_{r} = \pm1$. Since $b_r>0$ we get $b_r=1.$
\end{proof}
\noindent

So, $\pi_{1}^{-1}(z)\xlongrightarrow{\pi}\mathbb{P}^{m}$ is an isomorphism. This implies that the map $\Phi=(\pi, \pi_1): X\longrightarrow\mathbb{P}^{m}\times\mathbb{P}(F_{1},...,F_{r-1})$ has all fibres Spec $ k$, so $\Phi$ is an isomorphism. Since $\Phi$ is a map over $\mathbb{P}(F_{1},...,F_{r-1})$, so $\Phi: \mathbb{P}_{\mathbb{P}(F_{1},...,F_{r-1})}(q^*F_r)\longrightarrow\mathbb{P}_{\mathbb{P}(F_{1},...,F_{r-1})}(\mathcal{O}^{m+1})$ is an isomorphism of PGL$_{m+1}$-bundles, where $q: \mathbb{P}(F_{1},...,F_{r-1})\to \mathbb{P}^n$ is the natural projection. So $q^* F_r,$ hence $F_r$, is trivial up to line bundle twist.

So we have $ \mathbb{P}_{\mathbb{P}^m}(E_{1},...,E_{r})\cong \mathbb{P}^m \times   \mathbb{P}_{\mathbb{P}^n}(F_{1},...,F_{r-1}).$ By Corollary \ref{lemma:lemma 3.14}, we have  one of the two cases:\\\\
{\textbf{Case 1:}} Let $E_{1}\cong \mathcal{O}_{\mathbb{P}^m}^{m+1}\bigotimes L$ ( up to a permutation of $E_{1},E_{2},...,E_{r})$  and  $ \mathbb{P}_{\mathbb{P}^m}(E_{2},...,E_{r})$ is isomorphic to $\mathbb{P}_{\mathbb{P}^m}(F_{1},...,F_{r-1})$. By induction $E_{2},...,E_{r}$, $F_{1},...,F_{r-1}$ are trivial up to line bundle twists. So all $ E_{i}$'s, $F_{j}$'s are trivial up to line bundle twists.\\\\
\textbf{Case 2:} Let all $E_{i}$'s be trivial up to line bundle twists, $\mathbb{P}_{\mathbb{P}^n}(F_{1},...,F_{r-1})$ is isomorphic to $\mathbb{P}^{p_{1}}\times ...\times \mathbb{P}^{p_{r}}$. By Theorem \ref{theorem:theorem 3.17}, $F_{1},...,F_{r-1}$ are trivial upto line bundle twists. So, all $E_{i}$'s and $F_{j}$'s are trivial up to line bundle twists. This completes the proof using induction.}

\end{proof}

Now consider the following setup. let $m\neq n$ be natural numbers and $E_{1}$ and $F_{1}$ be vector bundles on $\mathbb{P}^{m}$ and $\mathbb{P}^{n}$ respectively and we have vector bundles $E_{2}$ and $F_{2}$ on $\mathbb{P}(E_{1})$ and $\mathbb{P}(F_{1})$ respectively. Suppose $\mathbb{P}_{\mathbb{P}(E_{1})}(E_{2})\cong\mathbb{P}_{\mathbb{P}{({F_{1}})}}(F_{2})$. Since $A^{\ast}(\mathbb{P}_{\mathbb{P}(E_{1})}(E_{2}))\cong A^{\ast}(\mathbb{P}^{m}\times\mathbb{P}^{\medspace\mathrm{rk}\medspace{E_{1}-1}}\times\mathbb{P}^{\medspace\mathrm{rk}\medspace E_{2}-1)}$ and $A^{\ast}(\mathbb{P}_{\mathbb{P}(F_{1})}(F_{2}))\cong A^{\ast}(\mathbb{P}^{n}\times\mathbb{P}^{\medspace\mathrm{rk}\medspace{F_{1}-1}}\times\mathbb{P}^{\medspace\mathrm{rk}\medspace F_{2}-1})$ as graded groups, Lemma \ref{lemma:lemma 4.1} shows that $\{m, \mathrm{rk}\medspace E_{1}-1,$ rk $E_{2}-1\}$ = $\{n,\mathrm{rk}\medspace F_{1}-1,\mathrm{rk}\medspace F_{2}-1\} $ as multisets.

Now let us consider the special case: rk$\medspace E_{2} = n+1$ and rk$\medspace F_{2} = m+1$. The
following theorem describes when  $ \mathbb{P}_{\mathbb{P}(E_{1})}(E_{2})$ can be isomorphic to  
$ \mathbb{P}_{\mathbb{P}(F_{1})}(F_{2})$. This can be regarded as a generalization of the $r = 2$ case of Theorem 4.4.

\begin{theorem} \label{theorem:theorem 4.5}
Let $m,n,r$ be positive integers with $m<n$. Let $ E_{1} $ and $F_{1}$  be vector bundles of rank $r+1$ over $ \mathbb{P}^m,\mathbb{P}^n$ respectively and   $ E_{2},F_{2}$ be vector bundles of rank $ n+1,m+1$ over $ \mathbb{P}(E_{1}), \mathbb{P}(F_{1})$ respectively. Suppose $ \mathbb{P}_{\mathbb{P}(E_{1})}(E_{2}) \cong \mathbb{P}_{\mathbb{P}(F_{1})}(F_{2})$. Then:  
\begin{enumerate}
\item[(i)] If $ r \neq m,n $ then all $ E_{i}$'s, $F_{j}$'s are trivial up to  line bundle twists. 
\item[(ii)] If $ r=m $, then $ E_{2}, F_{1}$ are trivial upto line bundle twists. There is a trivialization of $F_1$ such that if $ \mathbb{P}(F_{1})= \mathbb{P}^n \times \mathbb{P}^m \xlongrightarrow{ \pi} \mathbb{P}^m$  denotes the second  projection, then $ F_{2} = \pi^{\ast} E_{1} \bigotimes L$ for some line bundle $L$ on $ \mathbb{P}^n \times  \mathbb{P}^m .$
\item[(iii)] If $ r=n $ then $(ii)$ holds with $m$ replaced by n and $E_{i}$ replaced by $F_{i}$. 

\end{enumerate}
\end{theorem}
(In \textit{(ii)}, $ \mathbb{P}(F_{1})$ is identified with $\mathbb{P}^n \times \mathbb{P}^m$ via the stated trivialization of $F_1$.)
\begin{proof}
Let $p_{1}, p_{2},p, q_{1}$, $q_{2}$, and $q$ be the natural projections and $\phi$ an isomorphism as in the following diagram,
\begin{center}
\begin{tikzcd}[sep=large]
\mathbb{P}_{\mathbb{P}(E_{1})}(E_{2})\arrow[d,"p_2"] 
\arrow[bend right=60,swap]{ddrl}{p} \arrow[r,swap]{ur}{\sim}[swap]{\phi}
&{\mathbb{P}_{\mathbb{P}(F_{1})}(F_{2})} = X \arrow[bend left=60,swap]{ddrl}[swap]{q} \arrow[d,swap,"q_2"] \\  
\mathbb{P}(E_{1})  \arrow[d,"p_1"] &\mathbb{P}(F_{1}) \arrow[d,swap,"q_1"]\\ \mathbb{P}^{m} &\mathbb{P}^{n}
\end{tikzcd}
\end{center}
Let $t =\mathcal{O}_{\mathbb{P}^{n}}(1),$ $ u_{1} = \mathcal{O}_{\mathbb{P}(F_{1})}(1),$ $u_2= \mathcal{O}_{\mathbb{P}(F_{2})}(1),\,\  x= \mathcal{O}_{\mathbb{P}^{m}}(1),\,\ y_1=\mathcal{O}_{\mathbb{P}(E_{1})}(1),$ and $ y_2=\mathcal{O}_{\mathbb{P}(E_{2})}(1),$ all considered in $A^{1}(X)$ via pullback by the natural projections and $\phi$. Since $\{x, y_1, y_2\}$ and $\{t, u_1, u_2\}$ are both bases of $A^1(X)$ there is $A\in$ GL$_{3}{(\mathbb{Z})}$ such that 
\begin{gather}
\begin{pmatrix}
    x \\
    y_{1}\\
    y_{2}
\end{pmatrix}
=
A
\begin{pmatrix}
    t \\
    u_{1}\\
    u_{2}
\end{pmatrix}, 
\end{gather}
Let $a, b_{1}, b_{2}, \alpha, \beta_{1}, \beta_{2}\in\mathbb{Z}$ be such that 
$A$ looks like $ \begin{pmatrix}
    a & b_{1} & b_{2} \\
    \alpha & \beta_{1} & \beta_{2}\\
    \ast & \ast & \ast
\end{pmatrix}$ 
\,\ ; .

As in the proof of Theorem \ref{theorem:theorem 4.2}, there are homogeneous degree $r+1$ integral polynomials $f_1(Y_1,X),$ $ g_1(U_1,T)$ such that $Y_1^{r+1}, U_1^{r+1}$ have coefficients $1$ in $f_1, g_1$ respectively, and $f_1(1, -x)=(p\circ\phi^{-1} )^*c(E_1),$ $ g_1(1,-t)=q^*c(F_1).$ Also, there is homogeneous degree $m+1$ integral polynomial $g_2(U_2, U_1, T)$ with $U_2^{m+1}$ having coefficient $1$ and $g_2(1,-u_2,-t)=q_2^*c(F_2).$ Twisting  $E_1,$ $F_1$ by line bundles we can also make the following assumption, by the fact in point 4 of the notations and conventions section:
\\\\
\textbf{Assumption ($\ast$)}: If $f_1(Y_1, X)=(Y_1-cX)^{r+1}$ for some integer $c$, then $c=0.$  If $g_1(U_1, T)=(U_1-dT)^{r+1}$ for some integer $d$, then $d=0.$

For $z\in \mathbb{P}(F_1)$ we have $q_{2}^{-1}(z)\cong \mathbb{P}^{m}$ and $x|_{q_{2}^{-1}(z)} = \mathcal{O}_{q_{2}^{-1}(z)}(b_{2})$ is globally generated. So $b_{2}\geq 0$. We consider three cases separately.
\\\\
{\textbf{Case 1:}} \textbf{$b_2=0.$}

We have $x|_{q_{2}^{-1}(z)} = \mathcal{O}_{q_{2}^{-1}(z)}(b_{2})$ is trivial, for each $z\in \mathbb{P}(F_1)$. So, there is a map $\mathbb{P}(F_{1})\xlongrightarrow{\pi}\mathbb{P}^m$ making the following diagram commute,
\begin{center}
\begin{tikzcd}[sep=large]
\mathbb{P}_{\mathbb{P}(E_1)}(E_2) \arrow[d,"p_{2}"] & \cong & \mathbb{P}_{\mathbb{P}(F_1)}(F_2) \arrow[d,"q_2"] \\  \mathbb{P}(E_1)  \arrow[rd,"p_1"]   &    &\mathbb{P}(F_{1}) \arrow[dl,swap]{ur}{\pi} \\ & \mathbb{P}^m
\end{tikzcd}
\end{center}

If $r>m$, then each fiber of $\mathbb{P}(F_{1})\xlongrightarrow{q_{2}} \mathbb{P}^{n}$ must be contracted by $\pi$, as any map from $\mathbb{P}^{r}$ to $\mathbb{P}^{m}$ is constant for $r>m$. So $\pi$ factors as $\pi: \mathbb{P}(F_{1})\xlongrightarrow{q_{1}}\mathbb{P}^{n}\longrightarrow\mathbb{P}^{m}$. But the last map $\mathbb{P}^{n}\longrightarrow\mathbb{P}^{m}$ must be constant, as $n>m$. This shows $\pi\circ q_{2}$ is constant, hence $p$ is constant, which is a contradiction. Hence $r\leq m.$ 

For $z\in \mathbb{P}^{m}$, we have $p^{-1}(z)$ is a $\mathbb{P}^{n} -$bundle over $\mathbb{P}^{r}$, ${(\pi\circ q_{2})^{-1}(z)}$  a $\mathbb{P}^{m} -$bundle over $\pi^{-1}(z)$ and via $\phi$ we have $p^{-1}(z)\cong q^{-1}(z)$. So, $\pi^{-1}(z)$ is a smooth projective variety, and since $r\leq m$ and $m\neq n$ we have $\pi^{-1}(z)\cong\mathbb{P}^{n}$ and $r = m$ by Theorem \ref{theorem:theorem  3.9}. So, $\pi^{-1}(z)\cong\mathbb{P}^{n}$ for all $z\in\mathbb{P}^{m}.$ This, together with the fact that Br$(\mathbb{P}^{m}) = 0$, shows that $\pi$ is a Zariski locally trivial $\mathbb{P}^{n} - $bundle. So, $\mathbb{P}(F_{1})$ has a stucture of $\mathbb{P}^{m}-$bundle over $\mathbb{P}^{n}$, and also a structure of $\mathbb{P}^{n} -$bundle over $\mathbb{P}^{m}$. By Lemma \ref{lemma:lemma 3.7}(i)(Theorem A of \ref{9})we get $F_{1}$ is trivial with a line bundle twist and we can choose a trivialization of $F_1$ such that $\pi: \mathbb{P}(F_{1}) = \mathbb{P}^{n}\times\mathbb{P}^{m}\longrightarrow \mathbb{P}^{m}$ is the second projection. So, $\mathbb{P}(F_{1}) = \mathbb{P}_{\mathbb{P}^{m}}(\mathcal{O}_{\mathbb{P}^m}^{\bigoplus{n+1}})$ via $\pi.$ Now, by Lemma \ref{lemma:lemma 3.8}, $E_{2} = p_{1}^{\ast}\mathcal{O}^{n+1}\otimes M$ and  $F_{2} = \pi^{\ast}E_{1}\otimes L$ for line bundles $M, L$ on $\mathbb{P}(E_{1})$ and $\mathbb{P}(F_{1})$ respectively, hence $(ii)$ holds.
\\
\\
{\textbf{Case 2:}} $b_{2} = 1$.
\noindent
In this case ($ p\circ {\phi^{-1}}, q_{2}) = \Phi: X\longrightarrow \mathbb{P}^{m}\times\mathbb{P}(F_{1})$ is an isomorphism by Lemma 3.2, as each fibre is Spec $ k$. This is an isomorphism over $\mathbb{P}(F_{1})$, so $F_{2}$ is trivial with a line bundle twist. So, $\mathbb{P}_{\mathbb{P}(E_{1})}(E_{2})\cong \mathbb{P}^{m}\times\mathbb{P}(F_{1})$. As $m\neq n, $ by Theorem \ref{theorem:theorem 3.12} we get $\mathbb{P}_{\mathbb{P}^{m}}(E_{1}) = \mathbb{P}^{m}\times Z$ for some smooth projective variety $Z$. So, by Corollary \ref{lemma:lemma 3.14}, $E_{1}$ is trivial upto a line bundle twist, and $Z = \mathbb{P}^{r}$. The Theorem 3.12 also shows that $\mathbb{P}(F_{1})$ has a $\mathbb{P}^{n} - $bundle structure over $\mathbb{P}^{r}$.
\\
\\
{\textbf{Subcase 1:}} $r\neq n$.
We see that $\mathbb{P}_{\mathbb{P}^n}(F_{1})$ has a $\mathbb{P}^{r} -$bundle structure over $\mathbb{P}^{n}$ as well as a $\mathbb{P}^{n} -$bundle structure over $\mathbb{P}^{r}$. Since $r\neq n$, Lemma 3.7(i) shows $F_{1}$ is trivial up to a line bundle twist. So, $\mathbb{P}_{\mathbb{P}^{m}\times\mathbb{P}^{r}}(E_2)\cong \mathbb{P}^{m}\times\mathbb{P}^{n}\times\mathbb{P}^{r}$ as $E_1$, $F_{1}$ and $F_{2}$ are trivial upto line bundle twists. So by Corollary 3.13, $E_{2}$ is trivial up to a line bundle twist, so $(ii)$ holds. 
\\\\
{\textbf{Subcase 2:}} $ r = n$.
\noindent
Define a morphism $\tau_{1}$ by the following commutative diagram,
\begin{center}
\begin{tikzcd}
\mathbb{P}_{\mathbb{P}^{m}\times\mathbb{P}^{n}}(E_{2}) \arrow[r,swap]{ur}{\sim}[swap]{\phi} \arrow[d, "p_{1}"]
& \mathbb{P}^{m}\times\mathbb{P}_{\mathbb{P}^{n}}(F_{1}) \arrow[ld, "\tau_1"] \\
\mathbb{P}^{m}\times\mathbb{P}^{n}.
\end{tikzcd}
\end{center}
So, $\tau_1$ is a $\mathbb{P}^n-$bundle. By Lemma 3.5, there are smooth projective varieties $T,$ $Z$, an isomorphism $\psi:\mathbb{P}^{m}\times\mathbb{P}^{n}\to T\times Z$, surjective maps $h_1:\mathbb{P}^m\to T$, $h_2: \mathbb{P}_{\mathbb{P}^n}(F_1)\to Z$ such that $\psi\circ \tau_1=h_1\times h_2$. 
\begin{center}
\begin{tikzcd}
\mathbb{P}^{m} \times \mathbb{P}_{\mathbb{P}^n}(F_1)  \arrow[d, "\tau_1"] \arrow[rd, "h_1 \times h_2"] \\
\mathbb{P}^{m}\times \mathbb{P}^{n}   \arrow[r,swap]{ur}{\sim}[swap]{\psi} &  T \times Z
\end{tikzcd}
\end{center}
As $\tau_1$ is a $\mathbb{P}^n-$bundle, we see that $\psi\circ\tau_1$ is also a $\mathbb{P}^n-$bundle, so fibres of $\psi\circ\tau_1$ are $\mathbb{P}^n$. So, fibres of $h_1\times h_2$ are $\mathbb{P}^n$. Since $m\neq n$, we get that $h_1$ is an isomorphism and each fibre of $h_2$ is $\mathbb{P}^n$. So, $T=\mathbb{P}^m$, and $\mathbb{P}^{m}\times\mathbb{P}^{n}\cong T\times Z$ implies $Z=\mathbb{P}^n$ by Lemma \ref{lemma:lemma 3.6}. Since automorphisms of $\mathbb{P}^{m}\times\mathbb{P}^{n}$ preserves the factors, we see that $\tau_1=\eta\times\tau_2$ for some automorphism $\eta$ of $\mathbb{P}^m$, and a $\mathbb{P}^{n} -$bundle structure $\tau_{2}:\mathbb{P}_{\mathbb{P}^{n}}(F_{1})\longrightarrow\mathbb{P}^{n}$, which means $\tau_2$ the composition $\tau_2:\mathbb{P}_{\mathbb{P}^{n}}(F_{1})\cong\mathbb{P}_{\mathbb{P}^{n}}(F_{1}')\to \mathbb{P}^n$, where $F_1' $ is a rank $n+1$ vector bundle on $\mathbb{P}^n$, the first map is an isomorphism, not necessarily over $\mathbb{P}^n$, and the second map is the natural projection. Changing the trivialization of $F_2$ if necessary, we may assume that $\eta=id$. 
We have the following commutative diagram,  
\begin{center}
\begin{tikzcd}[sep=large]
\mathbb{P}_{\mathbb{P}^{m}\times\mathbb{P}^{m}}(E_{2})  \arrow[rr,swap]{ur}{\sim}[swap]{\phi} \arrow[d,"p_{1}"] &  &{\mathbb{P}^{m}\times\mathbb{P}_{\mathbb{P}^{n}}(F_{1})} \arrow[d] \\  \mathbb{P}^{m}\times\mathbb{P}^{n}  \arrow[rd,"\tau"]   &    &\mathbb{P}_{\mathbb{P}^{n}}(F_{1}) \arrow[dl,swap]{ur}{\tau_2} \\ & \mathbb{P}^n
\end{tikzcd}
\end{center}
where $\tau$ is the projection onto the second factor. By Lemma \ref{lemma:lemma 3.8}, $E_{2} = \pi^{\ast}F_{1}^{'}\otimes L_{1}$ for some line bundle $L_{1}$. We also have $\mathbb{P}_{\mathbb{P}^{n}}(F_{1}^{'})\cong\mathbb{P}_{\mathbb{P}^{n}}(F_{1})$. By Lemma \ref{lemma:lemma 3.7}(ii), $F_{1}^{'}\cong \psi^{\ast}F_{1}\otimes M$ for some automorphism $\psi$ of $\mathbb{P}^{n}$ and a line bundle $M$. So, $E_{2} = \pi^{\ast}\psi^{\ast}F_1\otimes L$ for some line bundle $L$ and automorphism $\psi$ of $\mathbb{P}^{n}$. By changing the trivialization of $E_{1}$, we can assume that $\psi$ is the identity morphism, so $(iii)$ holds.
\\
\\
{\textbf{Case 3:}} $b_{2}\geq 2$.

There are $Q_{1}({U_{2}, U_{1},T}), Q_{2}({U_{2}, U_{1},T})$  and $ R(U_{2}, U_{1}, T)\in\mathbb{Z}[U_{2}, U_{1}, T]$ such that $(aT+b_{1}U_{1}+b_{2}U_{2})^{m+1} = Q_{1}g_{1}+ Q_{2}g_{2}+R$, where $Q_{1}g_{1}, Q_{2}g_{2}$ and $R$ are homogeneous of degree $m+1$, deg$_{U_{1}}\medspace R\leq r$, deg$_{U_{2}}\medspace R\leq m,$ deg$_{T}\medspace R\leq m+1\leq n $. Since $(at+ b_{1}u_{1}+ b_{2}u_{2})^{m+1} = x^{m+1} = 0$ and $g_{1}(u_{1}, t) = g_{2}(u_{2}, u_{1}, t) = 0$, so we get $R(u_{2}, u_{1}, t) = 0$.

Since $\{t^{c_{1}}u_{1}^{c_{2}}u_{2}^{c_{3}}\medspace|\medspace c_{1}\leq n, c_{2}\leq r, c_{3}\leq m \}$ is linearly independent we get $R = 0$. So,
\begin{equation}
(aT+b_{1}U_{1}+b_{2}U_{2})^{m+1} = Q_{1}(U_2, U_1, T)g_{1}(U_1,T)+Q_{2}g_{2}(U_2, U_1, T).
\end{equation}
 By comparing the coefficients of $U_{2}^{m+1}$ in (8) we get $Q_2\equiv b_{2}^{m+1}$. So, 
\begin{equation}
(aT+ b_{1}U_{1}+ b_{2}U_{2})^{m+1} = Q_{1}g_{1}+b_{2}^{m+1}g_{2}.
\end{equation}
\\
If $r>m,$ we have deg$\medspace g_{1} = r+1 > m+1,$ so $Q_{1} = 0$. So,
\begin{equation*}
(\frac{b_{1}}{b_{2}}U_{1} + U_{2} + \frac{a}{b_{2}}T)^{m+1} = g_{2}\in \mathbb{Z}[U_{1}, U_{2}, T].     
\end{equation*}
By looking at the coefficients of $U_{1}^{m+1}$ and $T^{m+1}$ we get $\frac{b_{1}}{b_{2}}$ and $\frac{a}{b_{2}}\in \mathbb{Z}$, so $b_{2}$ divides $b_{1}$ and $a$. Since $A\in$ GL$_{3} (\mathbb{Z})$, we have gcd$(b_{1},b_{2},a) = 1$, so $b_{2}$ must be 1, which is a contradiction. So, $r\leq m$.

{\textbf{Subcase 1}} : $r \less m$

We have $$ 0= f_{1}(y_1,x)= f_{1}(\alpha t+ \beta_{1}u_{1}+ \beta_{2} u_{2}, at+ b_{1}u_{1}+b_{2}u_{2}) .$$ Hence $ f_{1}( \alpha  T + \beta_{1}U_{1}+\beta_{2}U_{2}, aT +b_{1}U_{1}+ b_{2}U_{2})$ is in the ideal generated by  \\$( T^{n+1}, g_{1}(U, T), g_{2}(U_{2},U_{1} ,T))$. Since deg $f_{1}=$ deg $g_{1} $ is smaller than both deg $g_{2}$ and deg $(T^{n+1})$, we see that  there is $\lambda \in \mathbb{Z}$ such that $$ f_{1}(\alpha T+ \beta_{1}U_{1}+ \beta_{2} U_{2}, aT+ b_{1}U_{1}+b_{2}U_{2})= \lambda g_{1} .$$

Write $$ f_{1}(Y_{1},X)= \prod_{i=1}^{r+1}(Y_{1}-\lambda_{i}X), \mathrm{ where} 
\medspace  \lambda_{i}\in  \mathbb{C} .$$
So, $$ \prod_{i=1}^{r+1} ((\alpha - \lambda_{i}a )T + ( \beta_{1}-\lambda_{i}b_{1} )U_{1} + (\beta_{2}-\lambda_{i}b_{2})U_{2})=\lambda g_{1}(U_{1},T)   $$ is independent of $ U_{2}$. This forces  $\beta_{2} = \lambda _{i}b_{2}\medspace$ for all  $i$. Since $ b_{2}\neq 0$, we see that all $ \lambda_{i}$'s are same and equal $ \frac{\beta_{2}}{b_{2}} .$ So, $ f_{1}(Y_{1},X) = ( Y_{1}-cX)^{r+1}$  for some $ c\in \mathbb{Q}$. Since $ f_{1} \in \mathbb{Z}[ X,Y]$, we get $ c\in \mathbb{Z}.$By assumption ($\ast$), $c=0$. So, $ f_{1} = Y_{1} ^{r+1}$. So $$ (\alpha T + \beta_{1}U_{1}+ \beta_{2}U_{2})^{r+1}= \lambda g_{1}(U_1,T) $$  is independent of $ U_{2}$ hence $ \beta_{2}=0  $. We have  $ \lambda g_{1}(U_{1},T)= (\alpha T+ \beta_{1}U_{1})^{r+1}$. Hence  $g_{1}(U_{1},T)= (U_{1}+c_{1}T)^{r+1} $  for some $ c_{1} \in \mathbb{Z}$, as coefficient of $U_1^{r+1}$ in $g_1$ is $1$ . By assumption ($\ast$) $c_1=0$, hence $g_{1}(U_1, T)= U_{1}^{r+1}$. So $ \lambda U_{1}^{r+1}= ( \alpha T + \beta _{1}U_{1})^{r+1}$, hence $\alpha =0 .$ By $(9)$  we have $$(aT+b_{1}U_{1}+b_{2}U_{2})^{m+1}= Q_{1}(\underline{U}, T) U_{1}^{r+1}+ b_{2}^{m+1}g_{2}(U_{2},U_{1},T). $$ Let $p \mid b_{2}$ be a prime. We get $$ (aT+b_{1}U_{1} )^{m+1}\equiv Q_{1} ( \underline{U} , T)U_{1}^{r+1} (mod \medspace p) .$$  So, $ U_{1}\mid (\Bar{a}T+\Bar{b_{1}}U_{1})^{m+1}$ in $\mathbb{F}_{p}[ U, T]$ where $ \Bar{a},\Bar{b_{1}} $ are the 
 residue class of $a,b_{1}$ mod $p$. This show $ \Bar{a}^{m+1},=0 $ so, $ \Bar{a}=0$. So if $ \Bar{A} $ is the image of $A$ in GL$_{3}(\mathbb{F}_{p})$, then we have 
 \begin{gather*}
\Bar{A}
=
 {\begin{pmatrix}
    0 & \ast & 0 \\
    0 & \ast & 0 \\
    \ast & \ast & \ast
\end{pmatrix}}
\end{gather*}
Clearly such an  $\Bar{A} $ can not be invertible,  a contradiction. \\\\
{\textbf{Subcase 2:}} $ r=m$\\ 
 In this case, we have $ Q_{1}\equiv \lambda $ for some $ \lambda \in \mathbb{Z} $ in $(9) $. So
\begin{equation} 
  (aT+b_{1}U_{1}+b_{2}U_{2})^{m+1}= \lambda g_{1}(U_{1}, T)+ b_{2}^{m+1}g_{2}(U_{2},U_{1},T).
  \end{equation}
The coefficient of $ U_{2}U_{1}^{m}$ on left hand side of (10) is $(m+1)b_{2}b_{1}^m$, and on right hand side this coefficient is divisible of by $ b_{2}^{m+1}$. 
So, $$ b_{2}^{m+1}\mid (m+1)b_{2}b_{1}^{m} \Rightarrow (m+1)(\frac{b_{1}}{b_{2}})^{m}\in \mathbb{Z}.$$   
If $ \frac{b_{1}}{b_{2}}= \frac{s_{1}}{s_{2}}$ with $ s_{1},s_{2} \in \mathbb{Z}$, $s_{2}\geq 1 $, gcd$ (s_{1},s_{2})=1$ then we get $ s_{2}^{m}\mid m+1$. So, $s_{2}^{m}\leq m+1$. 

 If $ m\geq 2 $, this forces that $s_{2}=1$, as otherwise $ m+1\geq 2^{m}$, which is impossible. So, $b_{2} \mid b_{1}$. Similarly, looking at the coefficient of $ U_{2}T^{m}$ in both sides of $(9)$ we get $ b_{2}\mid a$. Now $ A \in$ GL$_{3}(\mathbb{Z})$  forces $ b_{2}=1$, a contradiction. So $ m=1$.

  Note that $\mathbb{P}_{\mathbb{P}^1}(E_{1}), \mathbb{P}_{\mathbb{P}^{n}}(F_{1}) $ are of different dimensions, so they are not isomorphic.  Now by [\ref{10}, Theorem 2], we get $ \mathbb{P}_{\mathbb{P}^1}(E_{1}), \mathbb{P}_{\mathbb{P}^{n}}(F_{1}) $ are projective bundles over a smooth curve $ C$. Let

\hspace{35mm}
\xymatrix{ \mathbb{P}_{\mathbb{P}^1}(E_{1})  \ar[rd]_{p^{'}_{1}}   &  &\mathbb{P}_{\mathbb{P}^{n}}(F_{1})  \ar[ld]^\pi \\ & C}\\\\
be the projective bundle structures. An application of Theorem \ref{theorem:theorem 3.9} shows that $C \cong \mathbb{P}^1.$ Now by Lemma \ref{lemma:lemma 3.7}(i) (Theorem A of \ref{9}) $ F_{1}$ is trivial upto line bundle twist and we can choose trivialization of $ F_{1}$, so that $\pi: \mathbb{P}_{\mathbb{P}^n}(F_{1})= \mathbb{P}^n \times \mathbb{P}^1 \longrightarrow \mathbb{P}^1$ is the second projection. By [\ref{10}, Theorem 2], 
\begin{center}
\begin{tikzcd}
X \arrow[r, "p_{2}\circ \phi^{-1}"   ] \arrow[d, "q_{2}"]
&  \mathbb{P}_{\mathbb{P}^1}(E_{1})\arrow[d, "{p_{1}}'" ] \\
\mathbb{P}^n\times \mathbb{P}^1 \arrow[r, "\pi"]
& |[, rotate=0]|  \mathbb{P}^1 
\end{tikzcd}
\end{center}
is cartesian, so $ \mathbb{P}_{\mathbb{P}(E_{1})}(E_{2}) \cong \mathbb{P}(E_{1})\times \mathbb{P}^n $ over $ \mathbb{P}(E_{1})$. So, $ E_{2}$ is trivial with line bundle twists.

If $E_{1}$ is trivial up to line bundle twist, then $ \mathbb{P}_{\mathbb{P}^{1}\times \mathbb{P}^n}(F_{2})\cong \mathbb{P}^1 \times \mathbb{P}^n \times \mathbb{P}^1$, so by Corollary \ref{lemma: lemma 3.13}, $ F_{2}$ is also trivial up to line bundle twist. So,  $(i)$ holds.

If $ E_{1}$ is not trivial up to line bundle twist, then by the Theorem A of \ref{9}, $ p_{1}'= \psi \circ p_{1}$ for some automorphism $ \psi $ of $ \mathbb{P}^1$. So, since the above diagram is cartesian, so we get $ F_{2}= \pi^{\ast}\psi ^{\ast}E_{1}\bigotimes L $  for some line bundle. Changing the trivialization of the $ F_{1}$ we may assume that $ F_{2}= \pi^{\ast} E_{1}\bigotimes L$ for some line bundle $L$, so $(i)$ holds.\\\\
\textbf{Note:} Indeed, one can also check that if $(i)$ or $(ii)$ holds and $m=r=1 $, then $b_{2}=0 $ or 1. So the last case $ b_{2} \geq 2 $ can not hold.
\end{proof}
\end{section}

\section{Acknowledgement}We are grateful to Professor D.S. Nagaraj who urged us in several ways to do this work. Things we learnt from Suratno Basu in the National Centre of Mathematics workshop on Birational Geometry,  2023 summer, India helped us a lot in doing the work. We would also like to thank Jakub Witaszek, Souradeep Majumder, and János Kollár for giving valuable suggestions and references.\vskip 5mm
\noindent
\section{References}
\begin{enumerate}[label={[\arabic*]}]
\item \label{1} W.Fulton, \textit{Intersection theory, Second edition}, Springer-Verlag, Berlin, 1998.
\item \label{2} D.Eisenbud,  J.Harris, \textit{3264 and all that—a second course in algebraic geometry},
Cambridge University Press, Cambridge, 2016.
\item \label{3} R. Lazarsfeld {\it Some applications of the theory of positive vector bundles},  Lecture Notes in Math., 1092
Springer-Verlag, Berlin, 1984, 29–61.
\item \label{4}T. Fujita, {\it Cancellation problem of complete varieties}, Invent. Math. 64 (1981), no. 1,119–121.
\item \label{5} Kuroki, Shintarô, Suh, Dong Youp {\it Cohomological non-rigidity of eight-dimensional complex projective towers}, Algebra. Geom. Topol. 15 (2015), no. 2, 769–782.
\item  \label{6} C.Okonek, M.Schneider, H.Spindler,\textit {Vector bundles on complex projective spaces}, Progress in Mathematics, 3. Birkhäuser, Boston, MA, 1980.
\item \label{7}  K.Ueno, \textit{Algebraic geometry 2, Sheaves and cohomology.}, American Mathematical Society, Providence, RI, 2001.
\item \label{8} R. Hartshorne, \textit{Algebraic geometry}, Graduate Texts in Mathematics, No. 52. Springer-Verlag, New York-Heidelberg, 1977.
\item\label{9} Eiichi Sato,  \textit{Varieties which have two projective space bundle structures}, J. Math. Kyoto Univ. 25(3), 445-457 (1985).
\item \label{10} Gianluca Occhetta, Jaroslaw A. Wisniewski.: {\it On Euler-Jaczewski sequence and Remmert-Van de Ven problem for toric varieties}.\\ 
https://doi.org/10.48550/arXiv.math/0105166. 
\item \label{11} R.Lazarsfeld, \textit{Positivity in algebraic geometry I Classical setting:} line bundles and linear series, Ergebnisse der Mathematik und ihrer Grenzgebiete. 3. Folge. A Series of Modern Surveys in Mathematics 48, Springer-Verlag, Berlin, 2004.
\item \label{12} F.Hirzebruch,\textit {Topological methods in algebraic geometry}, Reprint of the 1978 edition, Classics in Mathematics, Springer-Verlag, Berlin, 1995.
\end{enumerate}

\pagebreak

\begin{flushleft}
{\scshape Indian Institute of Science Education and Research Tirupati, Rami Reddy Nagar, Karakambadi Road, Mangalam (P.O.), Tirupati, Andhra Pradesh, India – 517507.}

{\fontfamily{cmtt}\selectfont
\textit{Email address: ashimabansal@students.iisertirupati.ac.in} }
\end{flushleft}
\vspace{0.5mm}
\begin{flushleft}
{\scshape Fine Hall, Princeton, NJ 700108}.

{\fontfamily{cmtt}\selectfont
\textit{Email address: ss6663@princeton.edu} }
\end{flushleft}
\vspace{0.5mm}
\begin{flushleft}
{\scshape Indian Institute of Science Education and Research Tirupati, Rami Reddy Nagar, Karakambadi Road, Mangalam (P.O.), Tirupati, Andhra Pradesh, India – 517507.}

{\fontfamily{cmtt}\selectfont
\textit{Email address: shivamvats@students.iisertirupati.ac.in} }
\end{flushleft}

\end{document}